\documentclass{article}
\usepackage{graphicx} % Required for inserting images
\usepackage[a4paper, left=3.5cm, right=3.5cm, top=3cm]{geometry}
\usepackage[ english]{babel}
\usepackage{comment}
\usepackage{tabularx}
\usepackage{amsfonts}
\usepackage{amsmath, amssymb,amsthm}
\usepackage{mathtools}
\usepackage{relsize}
\usepackage{hyperref}
\usepackage{comment}
\usepackage[utf8]{inputenc}
\usepackage{enumerate, enumitem}
\usepackage{endnotes}

\newtheorem{theorem}{Theorem}[section]
\newtheorem{Def}[theorem]{Definition}

\newtheorem{lemma}[theorem]{Lemma}
\newtheorem{remark}[theorem]{Remark}

\newtheorem{corollary}[theorem]{Corollary}

\title{Low energy $\varepsilon$-harmonic maps into the round sphere}
\author{Andrew M. Roberts\endnote{School of Mathematics, University of Leeds, Leeds, LS2 9JT, United Kingdom\\ A.M.Roberts@leeds.ac.uk}}
\date{\today}

\begin{document}

\maketitle
\begin{abstract}
    In this paper we classify the low energy $\varepsilon$-harmonic maps from the surfaces of constant curvature with positive genus into the round sphere. We find that all such maps with degree $\pm1$ are all quantitively close to a bubble configuration with bubbles forming at special points on the domain with bubbling radius proportional to $\varepsilon^{1/4}$.
\end{abstract}

\section{Introduction}

The Dirichlet energy is one of the most widely studied objects in geometric analysis. Let $(M^2,g)$, $(N^n,h)$ be Riemannian manifolds and $u\in W^{1,2}(M,N)$. The Dirichlet energy is defined as follows
\begin{equation*}
    E[u]=\int_M|\nabla u|^2.
\end{equation*}
Critical points of this energy are known as harmonic maps. The Dirichlet energy functional is not a particularly nice functional, in the sense that regularity theory for it is difficult and it does not satisfy the Palais-Smale compactness conditions, so existence is also difficult.

In order to deal with the existence problem Sacks and Uhlenbeck introduced in \cite{Sacks_Uhlenbeck} the notion of $\alpha$-harmonic maps which arise as critical points of a perturbed energy functional. Regularity for these maps is much easier and the functional does indeed satisfy the Palais-Smale conditions, so in particular energy minimisers within homotopy classes exist. As $\alpha$ tends to 1 their perturbed functional returns to the Dirichlet energy so if we take a sequence of $\alpha_k$-harmonic maps $u_k$ with $k\rightarrow0$ we hope that this will converge to a fully harmonic map. This does indeed happen weakly in $W^{1,2}$ and in $C^\infty_{\text{loc}}$ away from a finite number of points. At these points we observe the standard bubbling phenomenon, when rescaled we obtain harmonic maps from the sphere.

In \cite{Lamm_epsilon} Lamm introduced the notion of $\varepsilon$-energy for $\varepsilon>0$ and $u\in W^{2,2}(M,N)$ defined as
\begin{equation}
    E_{\varepsilon}[u]=\int_M|\nabla u|^2+\varepsilon|\Delta u|^2
\end{equation}
where we further assume that $N$ is isometrically embedded into some Euclidean space. We call critical points of this energy $\varepsilon$-harmonic maps. This again satisfies Palais-Smale and has good regularity theory. Note that the Laplacian here is the extrinsic Laplacian obtained by viewing $u$ as a map $M\rightarrow N\hookrightarrow\mathbb{R}^l$ and so the energy depends on the embedding in Euclidean space chosen for $N$. Using the tension field instead would not give us any benefits as harmonic maps would still be critical points of our `intrinsic-$\varepsilon$-energy'. This offers the added complications from being a higher order term and the Euler-Lagrange equation becomes a 4th order equation, it however has the added benefit of linearity which offers some benefits in computations.

One of the questions that can be asked about these approximate harmonic maps is which maps can be obtained in the limit. When both the domain and the target the round 2 sphere, $\mathbb{S}^2$, all harmonic maps, up to reflection, are rational transformations. So in particular they must be M\"obius transformations in the degree 1 case. Though in fact it turns out not all of these can be reached by our approximate harmonic maps. Lamm, Malchiodi and Micallef \cite{LMM_limits},\cite{LMM_gap} showed that any degree $\pm1$ $\alpha$-harmonic maps with $\alpha$ sufficiently close to 1 and $\alpha$-energy sufficiently small must be a rotation, up to orientation.
The same result also holds for $\varepsilon$-harmonic maps, in \cite{horter_lamma_micallef} H\"orter, Lamm and Micallef showed that any degree $\pm1$ $\varepsilon$-harmonic map with $\varepsilon$ and $\varepsilon$-energy sufficiently small must be a rotation, up to orientation.  Gianocca also showed that a similar result holds for the Ginzburg-Landau approximation in \cite{gianocca}. For any $\gamma>0$ there exists a $\varepsilon_0$ such that all critical points $u:\mathbb{S}^2\rightarrow\mathbb{R}^3$ of the Ginzburg-Landau energy with energy $E^{GL}_\varepsilon[u]<8\pi-\gamma$ must be rotations, up to a scaling and reflection.

We will now study the case where the domain $\Sigma$ is some Riemannian surface of constant curvature and the target is again round $\mathbb{S}^2$. The key difference in this case is that there do not exist any degree $\pm1$ harmonic maps $u:\Sigma\rightarrow\mathbb{S}^2$. However we can certainly find degree $\pm1$ $\varepsilon$-harmonic maps by taking the energy minimiser within the homotopy class. By taking a sequence of degree 1 $\varepsilon$-harmonic maps with $\varepsilon$ decreasing to 0 and $\varepsilon$-energy decreasing to $4\pi$ we see that we must be in the case of bubble convergence. Due to the fact that our target is the round sphere we can use the energy identity and no neck results proved in \cite{Lamm_epsilon} and \cite{bayer_roberts}. These allow us to say that we are $W^{1,2}$ and $L^\infty$ close to a bubble configuration containing a constant limit map and a harmonic map $\mathbb{S}^2\rightarrow\mathbb{S}^2$. This gives us a qualitative picture of what $\varepsilon$-harmonic maps must look like when $\varepsilon$ is sufficiently small and $\varepsilon$-energy is sufficiently close to $4\pi$.

This same phenomenon occurs in the $\alpha$-harmonic case and in \cite{sharp_lowenergyalpha} Sharp shows quantitive results about this setup.
Indeed he shows that bubbles can only be blown at critical points of a specific function on the domain $\mathcal{J}$ and that the bubbling radius must be proportional to $\sqrt{\alpha-1}$, with the exact scale depending on the value of $\mathcal{J}$ at the blow up point. The proof of this relies on constructing a set of explicit singularity models $\mathcal{Z}$ which we know we are close to. Then by detailed analysis of the $\alpha$-energy on $\mathcal{Z}$ they obtain their results. This builds on the idea introduced by Malchiodi, Rupflin and Sharp in \cite{Malchiodi_Rupflin_Sharp}, and further developed by Rupflin in \cite{rupflin_lojasiewicz} and \cite{rupflin_lowenergylevelsharmonic}. 

In this paper we will show that this approach will also work for $\varepsilon$-harmonic maps and establish similar quantitive results in the low energy, degree $\pm1$ case.
To begin we define our domains as well as a special function.
\begin{Def}\label{def: sigma and J def}
    Let $(\Sigma, g)$ be a closed Riemannian surface of genus $\gamma\geq 1$ equipped with a metric $g$ of constant curvature zero when $\gamma=1$ and of constant curvature -1 otherwise. In the flat case we also impose $\textnormal{Area}_g(\Sigma) = 1$. Let $\{\phi_j\}$ be an arbitrary $L^2$-orthonormal basis of holomorphic one-forms on $\Sigma$ and define
    \begin{equation*}
        \mathcal{J} (a) := -2\pi c_\gamma \sum_{j} |\phi_j (a)|^2
    \end{equation*}
where $c_1=1$ and when $\gamma\geq 2$, $c_\gamma=4$. Note that one can show that $\mathcal{J}$ is independent of the choice of basis.
\end{Def}
This $\mathcal{J}$ is closely related to the Green's function on the surface as well as the Bergman kernel, see \cite{Malchiodi_Rupflin_Sharp} section 6.

As discussed we know that in our setting we must look like a harmonic bubble blown at a point.
This means that we will be able to parametrise our set of approximate bubbles $\mathcal{Z}$ by the following three variables. $a\in\Sigma$, the point at which our bubble is blown. $\lambda\in\mathbb{R}_{>1}$, the scale at which our bubble is blown, $1/\lambda$ will therefore be the bubbling radius. $R\in O(3)$, a rotation and choice of orientation of our bubble.
We will set
\begin{equation*}
    \pi_\lambda=\Big(\frac{2\lambda x}{1+\lambda^2|x|^2},\frac{1-\lambda^2|x|^2}{1+\lambda^2|x|^2}\Big)
\end{equation*}
then we will define $z=z_{R,\pi,a}\in C^\infty(\Sigma,\mathbb{S}^2)$ to look like $R\pi_\lambda$ in a neighbourhood of $a$ in local coordinates and be roughly constant away from $a$, see section 2 for the detailed construction. This means our singularity solutions will be a 6 dimensional non-compact manifold. In our neighbourhood of $a$ we have that for large $\lambda$, $\pi_\lambda$ roughly looks like
\begin{equation*}
    \Big(\frac{2x}{\lambda|x|^2},1\Big)\approx\Big(\frac{2}{\lambda}\frac{\partial}{\partial x}\log|x|,1\Big).
\end{equation*} 
The Greens function near $a$ in local coordinates looks like $\log|x|$ + some error term so smoothing out to the roughly constant section using the Greens function is a natural choice.

To motivate our results we will expand the $\varepsilon$-energy along $\mathcal{Z}$. We will have, similarly to the expansions in \cite{Malchiodi_Rupflin_Sharp} and \cite{sharp_lowenergyalpha}, for $z\in\mathcal{Z}$
\begin{equation*}
    E_\varepsilon[z]=4\pi-4\pi\mathcal{J}(a)\frac{1}{\lambda^2}+\frac{32\pi}{3c_\gamma}\varepsilon\lambda^2+\mathcal{O}(\frac{1}{\lambda^3})+\mathcal{O}(\varepsilon).
\end{equation*}
We shall choose $z$ to close to our critical point of $E_\varepsilon$, so differentiating with respect to $\lambda$ and $a$ should give us something small.
In particular one can show that
\begin{equation*}
    \partial_\lambda E_\varepsilon[z]=8\pi\mathcal{J}(a)\frac{1}{\lambda^3}+\frac{64\pi}{3c_\gamma}\varepsilon\lambda+\mathcal{O}(\frac{1}{\lambda^4})+\mathcal{O}(\frac{\varepsilon}{\lambda})
\end{equation*}
and, in the hyperbolic case,
\begin{equation*}
    \nabla_A  E_\varepsilon[z]=4\pi\nabla_A\mathcal{J}(a)\frac{1}{\lambda^2}+\mathcal{O}(\frac{1}{\lambda^3})+\mathcal{O}(\varepsilon\lambda).
\end{equation*}
Where $\nabla_A$ is the derivative of the base point $a$ in the direction $A$.
These align with na\"ively differentiating the energy. We will show that these quantities are indeed quantitively small which gives us information about the bubble scale and location.

The final element we need to introduce is the `$z$-norm' ($\|\cdot\|_z$) first introduced by Rupflin \cite{rupflin_lojasiewicz}. We already know that our maps are converging to a constant in an $L^2$ sense so we need to introduce a factor which blows up near the bubbles so we can see what happens near them. Further, Rupflin showed that the Dirichlet energy is non-degenerate with respect to this norm orthogonally to $\mathcal{Z}$ under certain conditions which will be very important. A difficulty occurs when using this norm in our setting as it does not involve any second order terms. We are able to overcome this by leveraging the smallness of $\varepsilon$.

\vspace{5mm}
We now state the main two theorems of the paper.
\begin{theorem}\label{thm: ratio of lambda to epsilon}
    Let $\Sigma$ be as in \ref{def: sigma and J def}. There exist $\varepsilon_0(\Sigma),\delta_0(\Sigma),C(\Sigma)>0$ such that if $0<\varepsilon<\varepsilon_0$, $u:\Sigma\rightarrow\mathbb{S}^2$ is a $\varepsilon$-harmonic map with $E_\varepsilon[u]\leq4\pi+\delta_0$ then either
    \begin{enumerate}
        \item $u$ has degree zero and $\|\nabla u\|_{L^\infty(\Sigma)}\leq C$, or
        \item  $u$ has degree $\pm 1$ and there exists $z(u)$ realising \textnormal{inf}$\{\|u-z\|_z: z\in\mathcal{Z}\}$, for any such $z(u)$, setting $a=a(z(u))$ and $\lambda=\lambda(z(u))$, we have
        \begin{equation*}
            |\mathcal{J}(a)\frac{1}{\lambda^4}+\frac{8}{3c_\gamma}\varepsilon|\leq C\varepsilon^{3/2}|\log\varepsilon|.
        \end{equation*}
    \end{enumerate}
\end{theorem}
\begin{theorem}\label{thm: J must be critical}
     Let $\Sigma$ be as in \ref{def: sigma and J def} with genus $\gamma\geq2$. There exist $\varepsilon_0(\Sigma),\delta_0(\Sigma),C(\Sigma)>0$ such that if $0<\varepsilon<\varepsilon_0$, $u:\Sigma\rightarrow\mathbb{S}^2$ is a $\varepsilon$-harmonic map with degree $\pm1$ and $E_\varepsilon[u]\leq4\pi+\delta_0$ then for any $z(u)$ realising \textnormal{inf}$\{\|u-z\|_z: z\in\mathcal{Z}\}$ with $a=a(z(u))$ we have
    \begin{equation*}
        \nabla_A\mathcal{J}(a)\leq C\varepsilon^{1/4}|\log\varepsilon|.
    \end{equation*}
\end{theorem}

These theorems immediately give us an explicit picture of what must happen in the limit.
\begin{corollary}
    Let $u_k:\Sigma\rightarrow\mathbb{S}^2$ be a sequence of $\varepsilon_k$-harmonic maps with $0<\varepsilon_k\downarrow0$ and $E_{\varepsilon_k}[u_k]\rightarrow\Lambda\leq4\pi$ then there exists a subsequence such that either
    \begin{enumerate}
        \item $u_k$ all have degree zero and converge smoothly to some degree zero harmonic map $u:\Sigma\rightarrow\mathbb{S}^2$, or
        \item $u_k$ all have degree 1 or all have degree -1 and converge in the standard bubbling sense to a bubble tree with the limit map constant and one bubble which is blown at a critical point $a_c$ of $\mathcal{J}$ at a bubbling rate of
        \begin{equation*}
            r_k=\Big(\frac{8\varepsilon_k}{3c_\gamma}|\mathcal{J}(a)|^{-1}\Big)^{1/4}.
        \end{equation*}
    \end{enumerate}
\end{corollary}
\begin{center}
    Acknowledgements
\end{center}
The author would like to thank his supervisor Ben Sharp for many helpful discussions. The author is funded by Engineering and Physical Sciences Research Council (EPSRC) - EP/W524372/1, Studentship 2927009.

\section{The space of approximate bubbles}

We start by explicitly defining $\mathcal{Z}$, our space of singularity models as in \cite{rupflin_lojasiewicz} and \cite{sharp_lowenergyalpha}.
With $(\Sigma,g)$ as in Definition \ref{def: sigma and J def} set $\iota=\frac{1}{2}\text{inj}(\Sigma,g)$, half of the injectivity radius. In the hyperbolic case set $\rho=\tanh(\iota)$ and note that for any $a\in\Sigma$ there exists an orientation preserving isometric isomorphism
\begin{equation*}
    F_a:(B_{2\iota}(a),g)\rightarrow(\mathbb{D}_\rho,\frac{4}{(1-|x|^2)^2}g_E)
\end{equation*}
with $F_a(a)=0$, $\mathbb{D}_\rho=\{x\in\mathbb{R}^2:|x|<\rho\}$ and $g_E$ is the Euclidean metric. In the flat case we simply take $\rho=2\iota$ and $F_a$ to be some choice of flat orientation preserving chart.
We will fix the Greens function on $\Sigma$ which solves
\begin{equation*}
    -\Delta_pG(p,q)=2\pi\delta_q-\frac{2\pi}{\text{Vol}(\Sigma)}.
\end{equation*}
From now on work with local coordinates, for $p,q\in B_{2\iota}(a)$ we write $x=F_a(p)$ and $y=F_a(q)$. In these coordinates we can locally write the Greens function as
\begin{equation}
    G(p,q)=G_a(x,y)=-\log|x-y|+J_a(x,y),
\end{equation}
where $J_a$ is some smooth function, dependent on the choice of $F_a$.
Set $r=\rho/4$, then fix some $\phi\in C_c^\infty(\mathbb{D}_{2r},[0,1])$ radial with $\phi\equiv 1$ on $\mathbb{D}_r$. Now define for $\lambda\geq1$
\begin{equation*}
    \tilde{z}_{\lambda,a}(p)=\begin{cases}
        \hat{z}_{\lambda,a}(F_a(p))&\text{if }p\in B_\iota(a)\\
        (\frac{2}{\lambda}(\partial_{q^1}G(p,a)-\partial_{y^1}J_a(0,0)),\frac{2}{\lambda}(\partial_{q^2}G(p,a)-\partial_{y^2}J_a(0,0)),-1)&\text{if }p\notin B_\iota(a)
    \end{cases}
\end{equation*}
where, in our local coordinates
\begin{align*}
    \hat{z}_{\lambda,a}(x)=&\phi(x)\big(\pi_\lambda(x)+(\frac{2}{\lambda}(\nabla_yJ_a(x,0)-\nabla_yJ_a(0,0)),0)\big)\\
    &+(1-\phi(x))(\frac{2}{\lambda}(\nabla_yG_a(x,0)-\nabla_yJ_a(0,0)),-1)
\end{align*}
with $\pi_\lambda(x)=\big(\frac{2\lambda x}{1+\lambda^2|x|^2},\frac{1-\lambda^2|x|^2}{1+\lambda^2|x|^2}\big)$, the stereographic projection at scale $\lambda$.
Finally set
\begin{align*}
    z_{\lambda,a}=&P_{\mathbb{S}^2}(\tilde{z}_{\lambda,a})=\frac{\tilde{z}_{\lambda,a}}{|\tilde{z}_{\lambda,a}|}\\
    \mathcal{Z}=&\{Rz_{\lambda,a}|a\in\Sigma,R\in O(3),\lambda>1\}.
\end{align*}
Also, in our local coordinates, define
\begin{equation*}
    j(x)=j_{\lambda,a}(x):=(\frac{2}{\lambda}(\nabla_yJ_a(x,0)-\nabla_yJ_a(0,0)),0).
\end{equation*} 
So in $\mathbb{D}_{r}$ we have $z=P_{\mathbb{S}^2}(\pi+j)$.
We will also frequently use the notation for $t\leq 2r$
\begin{equation*}
    U_t:=F_a^{-1}(\mathbb{D}_t).
\end{equation*}

This gives us our family of singularity models which look like a single bubble forming at a point. We will now formally define the norm discussed in the introduction as introduced by Rupflin in \cite{rupflin_lojasiewicz}.
\begin{Def}
    Given $V,W\in W^{1,2}(\Sigma,\mathbb{R}^3)$ and $z\in\mathcal{Z}$ define an inner product
    \begin{equation*}
        \langle V,W\rangle_z=\int_\Sigma\nabla V\cdot\nabla W+\rho^2_zV\cdot W\textnormal{d}\Sigma
    \end{equation*}
    where
    \begin{equation*}
        \rho_z(p):=\begin{cases}
                        \frac{\lambda}{1+\lambda^2\textnormal{d}_g(p,a)^2} &\text{if }p\in B_\iota(a)\\
                        \frac{\lambda}{1+\lambda^2\iota^2} &\text{if }p\in\Sigma\setminus B_\iota(a).
                    \end{cases}
    \end{equation*}
    As usual we write $\|V\|_z:=\langle V,V\rangle_z^{1/2}$.
\end{Def}
The factor $\rho$ rescales the $L^2$ norm around the bubbling point to the bubble scale, which allows detailed analysis near it.

We also mention a bound shown in \cite{rupflin_lojasiewicz}, which allows us to turn on $L^2$ bounds or an integral around a boundary into a `$z$-norm' bound at some cost in $\lambda$. 
\begin{remark}\label{remark: log estimate}[\cite{rupflin_lojasiewicz} (2.28) and Appendix B]
    For any $s\in[1,\infty)$ there exists $K(s,\Sigma)$ such that
    \begin{equation*}
        |\int_{\partial B_\iota(a)}v\textnormal{d}\Sigma|+\|v\|_{L^s(\Sigma)}\leq K(\log\lambda)^{1/2}\|v\|_z.
    \end{equation*}
\end{remark}

\section{Preliminary computations}

To begin with we expand the $\varepsilon$-energy along $\mathcal{Z}$. For this and many other of our results we will require very detailed estimates on $z$ and its derivatives. We have deferred these to Appendix \ref{section: bounds on z}.

\begin{lemma}\label{Lemma: Raw energy bound}
    For $z\in \mathcal{Z}$ and $0<\varepsilon<1$ we have
    \begin{equation*}
        E_\varepsilon[z]=4\pi-4\pi\mathcal{J}(a)\frac{1}{\lambda^2}+\frac{32\pi}{3c_\gamma}(\varepsilon\lambda^2)+\mathcal{O}(\frac{1}{\lambda^3})+\mathcal{O}(\varepsilon).
    \end{equation*}
\end{lemma}

\begin{proof}
For the Dirichlet part see \cite{Malchiodi_Rupflin_Sharp} or appendix \ref{section: proof of raw energy bound}.
\begin{equation*}
    E[z]=4\pi-4\pi\mathcal{J}(a)\frac{1}{\lambda^2}+\mathcal{O}(\frac{1}{\lambda^3}).
\end{equation*}
For the biharmonic part we have, using \eqref{eq: z bounds away from D_r},
\begin{align*}
    \varepsilon\int_{\Sigma}|\Delta z|^2=&\varepsilon\int_{\mathbb{D}_{r}}|\Delta^g z|^2\textnormal{d}x_g+\mathcal{O}(\varepsilon).
\end{align*}
Then from \eqref{eq: Delta z in the ball}
\begin{equation*}
    |\Delta^g z|^2=\frac{1}{c_\gamma^2}|\Delta \pi_\lambda|^2+\mathcal{O}(\frac{\lambda^2}{(1+\lambda^2|x|^2)^2}),
\end{equation*}
then using \eqref{eq: exact pi derivatives},
\begin{align*}
\begin{split}
    \int_{\mathbb{D}_r}\frac{1}{c_\gamma^2}|\Delta \pi_\lambda|^2\text{d}x_g=&
    \int_{\mathbb{D}_r}\frac{64\lambda^4}{(1+\lambda^2|x|^2)^4}(\frac{1}{c_\gamma}+\mathcal{O}(|x|^2))dx=
    \frac{64\lambda^2\pi}{3c_\gamma}+\mathcal{O}(1)\\
    \int_{\mathbb{D}_r}\mathcal{O}(\frac{\lambda^2}{(1+\lambda^2|x|^2)^2})\text{d}x_g=&\mathcal{O}(1).
\end{split}
\end{align*}
Combining these integrals completes the proof.

\end{proof}

We also note here that, for $u\in W^{2,2}(\Sigma,\mathbb{S}^2)$ and $V,W\in W^{2,2}(\Sigma,\mathbb{R}^3)$ with $V(u(x)),W(u(x))\in T_{u(x)}\mathbb{S}^2$ almost everywhere, we have
\begin{equation}
\begin{split}\label{eq: variation expansions}
    \text{d}E_\varepsilon(u)[V]=&\int_\Sigma\nabla u\cdot\nabla V+\varepsilon\int_\Sigma\Delta u\cdot\Delta V\\
    \text{d}^2E_\varepsilon(u)[V,W]=&\int_\Sigma\nabla V\cdot\nabla W-|\nabla u|^2(V\cdot W)
    +\varepsilon\int_\Sigma\Delta V\cdot\Delta W-\Delta u\cdot\Delta(u(V\cdot W))
\end{split}
\end{equation}

We mention the energy identity and no-neck results for $\varepsilon$-harmonic maps into the round 2-sphere, noting again that the sphere must be embedded equivariantly into $\mathbb{R}^3$. These allow us to say that our $\varepsilon$-harmonic maps do look like our models in an $W^{1,2}$ sense and in an $L^\infty$ sense.
\begin{theorem}\label{theorem: bubble convergence}[\cite{Lamm_epsilon} Theorem 1.1, \cite{bayer_roberts} Theorem 1.5]
    Let $\Sigma$ be as in Definition \ref{def: sigma and J def} and $\mathbb{S}^2\hookrightarrow\mathbb{R}^3$ be the standard embedding. If $(u_k)\in W^{2,2}(\Sigma,\mathbb{S}^2)$ is a sequence of $\varepsilon_k$ harmonic maps with $0<\varepsilon_k\downarrow0$ and $E_{\varepsilon_k}[u_k]$ uniformly bounded, then along some subsequence there exist a smooth harmonic map $u_\infty:\Sigma\rightarrow\mathbb{S}^2,$ a finite non negative number of smooth harmonic maps $(\omega^j):\mathbb{R}^2\rightarrow\mathbb{S}^2$ and corresponding sequences $(a^j_k,r^j_k)$ with $a^j_k\rightarrow a^j\in \Sigma$ and $r^j_k\rightarrow 0$ such that
    \begin{enumerate}
        \item $E_{\varepsilon_k}[u_k]\rightarrow E[u_\infty]+\sum\limits_{k}E[\omega_k],$
        \item $\|u-u_\infty-\phi^j(\cdot)\sum\limits_{k}(\omega^j((r^j_k)^{-1}F_{a^j_k}(\cdot))-\omega^j(\infty))\|_{L^\infty(B_\iota(a^j))}\quad$ for each $j$
    \end{enumerate}
    where $\phi^j\in C^\infty(\Sigma)$ is defined by $\phi^j(p)=\phi(F_{a^j}(p))$ for some fixed $\phi\in C_c(\mathbb{D}_t,[0,1])$ with $\phi\equiv 1$ on $\mathbb{D}_{t/2}$, for $t$ fixed with $t<\rho/2,\textnormal{min}_{j_1\neq j_2}\{|a^{j_1}-a^{j_2}|/2\}$.
\end{theorem}

In \cite{Lamm_epsilon} Lamm showed a low energy regularity theorem for $\varepsilon$ harmonic maps which we recall. Due to our low energy setting this will entail pointwise estimates on the neck region.
\begin{lemma}\label{lemma: regularity theorem}[\cite{Lamm_epsilon} Corollary 2.10]
    There exist $\delta_0>0$ and $C>0$ such that if $u_\varepsilon\in C^\infty(\Sigma,\mathbb{S}^2)$ is $\varepsilon$-harmonic with $E_\varepsilon(u_\varepsilon,B_{32R}(x_0))<\delta_0$ for some $x\in\Sigma$ and $R>0$. Then for all $\varepsilon>0$ sufficiently small and any $k\in \mathbb{N}$
    \begin{equation*}
        \sum_{i=1}^kR^i\|\nabla^iu_\varepsilon\|_{L^\infty(B_R(x_0))}\leq C\sqrt{E_\varepsilon(u_\varepsilon,B_{32R}(x_0))}.
    \end{equation*}
\end{lemma}

We now start by clarifying the degree 0 case.
\begin{lemma}\label{lemma: 0 degree case}
    There exists $\varepsilon_0(\Sigma)>0,\delta_0(\Sigma)>0,C(\Sigma)>0$ such that if $0<\varepsilon<\varepsilon_0$, $u:\Sigma\rightarrow \mathbb{S}^2$ is a degree zero $\varepsilon$-harmonic map with $E_\varepsilon[u]\leq 4\pi+\delta_0$ then $\|\nabla u\|_{L^\infty(\Sigma)}\leq C$
\end{lemma}

\begin{proof}
    Suppose not, then there exists a sequence $(u_k)$ of $\varepsilon_k$-harmonic maps with degree zero with $\varepsilon_k\downarrow0$ and $E_{\varepsilon_k}[u_k]\rightarrow 4\pi$, but $\|\nabla u\|_{L^\infty(\Sigma)}\rightarrow\infty$. This means that we cannot have smooth convergence of $(u_k)$, so along a subsequence we must have bubble convergence to a limit with at least one bubble. By consideration of energy we must have a bubble of degree $\pm1$ and the limit map $u_\infty$ constant. But due to the no neck property in Theorem \ref{theorem: bubble convergence} this convergence must preserve homotopy which contradicts $u_k$ being degree 0.
\end{proof}

In the degree 1 case, as discussed, any $\varepsilon$-harmonic map should be $W^{1,2}$ and $L^\infty$ close to something in $\mathcal{Z}$. We prove this and establish a quantitive result.

\begin{lemma}\label{lemma: z is close!}
    For any $\gamma>0$ there exist $\varepsilon_0(\Sigma),\delta_0(\Sigma)>0$ such that if $0<\varepsilon<\varepsilon_0$ and $u:\Sigma\rightarrow\mathbb{S}^2$ a degree one $\varepsilon$ harmonic map with $E_\varepsilon[u]\leq4\pi+\delta_0$ then we have the following
\begin{enumerate}
    \item $\|\nabla u\|_{L^\infty}>\frac{1}{\gamma},$
    \item $\exists z\in \mathcal{Z}$ such that $\|u-z\|_{W^{1,2}}+\|u-z\|_{L^\infty}\leq \gamma,$
    \item There exists $z\in\mathcal{Z}$ with $\|u-z\|_z=\inf_{\eta\in\mathcal{Z}}\{\|u-\eta\|_\eta\}$. For this $z$ we then have $\|u-z\|_{z}+\|u-z\|_{L^\infty}\leq \gamma$.
\end{enumerate}
\end{lemma}

\begin{proof}

Take any collection of $u_k$, degree 1 $\varepsilon_k$-harmonic maps with $\varepsilon_k\downarrow0$ and $E_{\varepsilon_k}[u_k]\downarrow 4\pi$. If $\|\nabla u_k\|_{L^\infty}+\varepsilon_k\|\Delta u_k\|_{L^\infty}\leq\frac{1}{\gamma}$ for infinitely many $k$ then Lemma \ref{lemma: regularity theorem} implies that along a subsequence we must have smooth convergence to a smooth degree one harmonic map $u_\infty:\Sigma\rightarrow\mathbb{S}^2$ with energy $4\pi$, which we know does not exist. This means we must be in the case of bubble convergence to precisely one bubble. So we must have $\|\nabla u_k\|_{L^\infty}+\varepsilon_k\|\Delta u_k\|_{L^\infty}>\frac{1}{\gamma}$ for any $k$ large enough. Results from \cite{Lamm_epsilon} now imply that $\varepsilon_k\|\Delta u_k\|_{L^\infty}\leq C\frac{\varepsilon_k}{r_k^2}\rightarrow 0$, where $r_k$ is the bubble scale of $u_k$. It is clear that we must be able to find $\varepsilon_0,\delta_0$ such that \textit{1.} follows.

By the energy identity (Theorem \ref{theorem: bubble convergence}), we must have that the limit map, $u_\infty$, is constant and that the bubble, $\omega$, is a harmonic map $\mathbb{S}^2\rightarrow\mathbb{S}^2$ with energy $4\pi$. This means that, after possibly rotating the domain and target, $\omega$ must be of the form $\pi_{\mu}$ for some $\mu$ fixed, meaning the stereographic projection at scale $\mu$. Now consider $z_k:=z_{\mu r_k^{-1},a_k}$ where $r_k$ is the bubbling radius and $a_k$ the bubbling point of $u_k$. Now $z_k$ must also converge in a $W^{1,2}$ and $L^\infty$ sense to the same bubble configuration as $u_k$. So using both the energy identity and the no neck property we must have
\begin{equation*}
    E[u_k-z_k]+||u_k-z_k||_{L^\infty}\rightarrow0.
\end{equation*}
and so \textit{2.} must hold.

Taking the $z$ from \textit{2.} the bubble scale converges to 0 so $\mu r_k^{-1}\rightarrow\infty$, this means that for $k$ large we can apply Lemma 4.2 from \cite{rupflin_lojasiewicz}. For $k$ large enough, there exist $\zeta_{k}\in\mathcal{Z}$ minimising $\| u_k-\eta\|_\eta$ over $\eta\in\mathcal{Z}$ with $\|u_k-\zeta_{k}\|_{L^\infty}\rightarrow0$. We then have $\| u_k-\zeta_k\|_{\zeta_k}\leq C\| u_k-z_k\|_{z_k}\rightarrow0$ completing the proof.

\end{proof}

Now we have shown that $u$ must be close to some $z$ we can start to pass some of the properties from $z$ onto $u$. The following lemma says that we can in some sense quantitively think of $\lambda$ as the inverse of the bubbling radius. These bounds would follow easily from Lemma \ref{lemma: regularity theorem} if $\lambda$ were replaced by the inverse of some choice of the bubbling radius.

\begin{lemma}\label{lemma: L^ infintity and 2 bounds}

There exists $\varepsilon_0(\Sigma),\delta_0(\Sigma)$ such that if $0<\varepsilon<\varepsilon_0$, $u:\Sigma\rightarrow \mathbb{S}^2$ is a degree 1 $\varepsilon$-harmonic map with $E_\varepsilon[u]-4\pi\leq \delta_0$ and $z\in \mathcal{Z}$ with $\|u-z\|_z^2\leq \delta_0$ then
\begin{enumerate}
    \item $\|\nabla^k u\|_{L^\infty(\Sigma)}\leq C(k,\Sigma)\lambda^k$ for $k\geq 0,$
    \item $\|\nabla^k u\|_{L^2(\Sigma)}\leq C(k,\Sigma)\lambda^{k-1}$ for $k\geq 1,$
    \item $\varepsilon\lambda^2\leq C(\Sigma)(E_\varepsilon[u]-4\pi)\leq C(\Sigma)\delta_0.$
\end{enumerate}
\end{lemma}

\begin{proof}

In $\Sigma\setminus U_{r/4}$, where $r$ is as in the definition of $z$, we have $|\nabla z|\leq C$, with $C$ independent of $z$. We can then choose some $\gamma>0$  small enough such that $\|\nabla z\|_{L^2(B_\gamma(p))}^2\leq \delta_0$ for any $p\in\Sigma\setminus U_{r/2}$. We then have
\begin{equation*}
    E_\varepsilon[u;B_\gamma(p)]\leq 2\|\nabla z\|_{L^2(B_\gamma(p))}^2+2\|u-z\|_z^2 +\varepsilon\|\Delta u\|_{L^2(\Sigma)}^2\leq5 \delta_0
\end{equation*}
Now we can take $\delta_0,\varepsilon_0$ small enough such that we may apply Lemma \ref{lemma: regularity theorem}. This means that $\|\nabla^k u\|_{L^\infty(B_{\gamma/32}(p))}\leq C\gamma^{-k}\leq C(k,\Sigma)$, giving our $L^\infty$ bound on $\Sigma\setminus U_{r/2}$. 
In $U_r$ we have $|\nabla z|\leq C\frac{\lambda}{1+\lambda^2|x|^2}$ using \eqref{eq: z remainders}. By setting $\hat{z}(x)=z(x/\lambda)$, defined in our local coordinates on $\mathbb{D}_{r\lambda}(0)$, we again have $|\nabla \hat{z}|\leq C(\Sigma)$. So we may again choose some $\gamma>0$ small enough such that $\|\nabla \hat{z}\|_{L^2(\mathbb{D}_{\gamma}(y))}^2\leq \delta_0$ for any $y\in \mathbb{D}_{\lambda r/2}(0)$. So for any $x\in \mathbb{D}_{r/2}$ we have
\begin{equation*}
    E_\varepsilon[u;\mathbb{D}_{\lambda^{-1}\gamma}(x)]\leq 2\|\nabla z\|_{L^2(\mathbb{D}_{\lambda^{-1}r}(x))}^2+2\|u-z\|_z^2 +\varepsilon\|\Delta u\|_{L^2(\Sigma)}^2\leq 5\delta_0
\end{equation*}
by scale invariance of energy.
Then by Lemma \ref{lemma: regularity theorem} we have $\|\nabla^k u\|_{L^\infty(\mathbb{D}_{\lambda^{-1}\gamma /32}(x))}\leq C\gamma^{-k}\lambda^k\leq C(k,\Sigma)\lambda^k$, completing our $L^\infty$ bound.

For the $L^2$ bound now consider the annulus $A=\mathbb{D}_r\setminus \mathbb{D}_{T\lambda^{-1}}$ for some $T>0$ to be decided. Using \eqref{eq: z remainders} we get
\begin{equation*}
    \|\nabla z\|_{L^2(A)}^2\leq C\int_{B_r(0)\setminus B_{T\lambda^{-1}}(0)}\frac{\lambda^2}{(1+\lambda^2|x|^2)^2}\leq C\int_{\mathbb{R}^2\setminus B_{T}(0)}\frac{1}{(1+|x|^2)^2}\leq \delta_0
\end{equation*}
by fixing $T(\Sigma)$ large enough. As before this implies that $E_\varepsilon[u;A]\leq 5 \delta_0$. Now for any $x\in \mathbb{D}_{r/2}\setminus \mathbb{D}_{2T\lambda^{-1}}$ we have that $\mathbb{D}_{|x|/2}(x)\subset A$ and so by Lemma \ref{lemma: regularity theorem} we have $\|\nabla^k u\|_{L^\infty(B_{|x|/64}({x}))}\leq C(k,\Sigma)|x|^{-k}$. This then gives
\begin{align*}
    \|\nabla^k u\|^2_{L^2(\Sigma)}&\leq C\bigg(\int_{\Sigma\setminus U_{R/2}(a)} + \int_{B_{R/2}(a)\setminus B_{2T\lambda^{-1}}(a)}\frac{1}{|x|^{2k}}+\int_{ B_{2T\lambda^{-1}}(a)}\lambda^{2k}\bigg)\\
    &\leq C\lambda^{2k-2}
\end{align*}
completing the $L^2$ bound. The $L^2$ bound then immediately implies that
\begin{equation}\label{eq: Delta into z norm}
    \|\Delta (u-z)\|_{L^2}^2=\int_\Sigma\nabla(u-z)\cdot\nabla\Delta(u-z)\leq\|u-z\|_z\|\nabla\Delta(u-z)\|_{L^2}\leq C\lambda^2\|u-z\|_z
\end{equation}
which allows us to write, using the computations from Lemma \ref{Lemma: Raw energy bound},
\begin{equation*}
    E_\varepsilon[u]-4\pi\geq \varepsilon\|\Delta u\|_{L^2}^2\geq \varepsilon(\frac{1}{2}\|\Delta z\|_{L^2}^2-\|\Delta (u-z)\|_{L^2}^2)\geq \varepsilon(C_1\lambda^2-C_2\lambda^2\delta_0^{1/2})
\end{equation*}
which completes the final bound by taking $\delta_0$ small enough.
\end{proof}

\eqref{eq: Delta into z norm} will be essential later on as it allows us to bound 2nd order terms in terms of the $z$-norm.

\begin{lemma}\label{lemma: non-degeneracy}
    There exist $\varepsilon_0(\Sigma),\delta_0(\Sigma),c_0>0$ such that if $0<\varepsilon<\varepsilon_0$, $u:\Sigma\rightarrow\mathbb{S}^2$ is a degree one $\varepsilon$-harmonic map with $E_\varepsilon\leq4\pi+\delta_0$ and $z\in\mathcal{Z}$ realises $\|u-z\|_z=\inf_{\eta\in\mathcal{Z}}\{\|u-\eta\|_\eta\}$ them setting $w=u-z$ and $W=\text{d}P(z)[w]$ gives
    \begin{equation*}
        \textnormal{d}^2E_\varepsilon(z)[W,W]\geq c_0\|w\|_z^2
    \end{equation*}
\end{lemma}

\begin{proof}

We have from \eqref{eq: variation expansions}
\begin{equation*}
    \textnormal{d}^2E_\varepsilon(z)[W,W]=\textnormal{d}^2E(z)[W,W]+\varepsilon\int|\Delta W|^2-(\Delta^2z\cdot z)|W|^2
\end{equation*}
From Lemma 3.2 in \cite{rupflin_lojasiewicz} we have that, when $\lambda(z)$ is large enough, all eigenvalues of $\text{d}^2E$ are bounded away from 0 on $T_z\mathcal{Z}^{\perp_z}$, the orthogonal complement of $T_z\mathcal{Z}$ with respect to the $\langle\cdot,\cdot\rangle_z$ inner product. We can guarantee $\lambda$ is large enough by Lemma \ref{lemma: z is close!} and Lemma \ref{lemma: L^ infintity and 2 bounds}. We also know that these eigenvalues must be positive as the limiting bubble is stable. Now from Lemma 4.3 in \cite{rupflin_lojasiewicz} we have that $\|W^{\top_{T_z\mathcal{Z}}}\|\leq C\|w\|_{L^\infty}\|w\|_z$ which gives, using Lemma \ref{lemma: z is close!} to ensure $\|w\|_{L^\infty}$ small enough,
\begin{equation*}
    \|W^{\bot_{T_z\mathcal{Z}}}\|^2_z\geq \|W\|^2_z-C\|w\|^2_{L^\infty}\|w\|^2_z\geq \frac{1}{4}\|w\|^2_z
\end{equation*}
also using \eqref{eq: w and W equiv.}.
We also note that
\begin{equation*}
    d^2E(z)[A,B]=\int_\Sigma\nabla A\cdot\nabla B-|\nabla z|^2(A\cdot B)\leq C\|A\|_z\|B\|_z
\end{equation*}
Let $c_0$ be the lowest eigenvalue of $E(z)$ on $T_z\mathcal{Z}^{\perp_z}$.
We can then calculate, using Young's inequality on the second term,
\begin{align*}
    d^2E(z)[W,W]=&d^2E(z)[W^{\bot_{T_z\mathcal{Z}}},W^{\bot_{T_z\mathcal{Z}}}]+2d^2E(z)[W^{\bot_{T_z\mathcal{Z}}},W^{\top_{T_z\mathcal{Z}}}]+d^2E(z)[W^{\top_{T_z\mathcal{Z}}},W^{\top_{T_z\mathcal{Z}}}]\\
    \geq& \frac{c_0}{2}\|W^{\bot_{T_z\mathcal{Z}}}\|^2_z- C\|W^{\top_{T_z\mathcal{Z}}}\|^2_z\\
    \geq& \frac{c_0}{16}\|w\|^2_z
\end{align*}
by taking $\varepsilon_0,\delta_0$ small enough.
Using \eqref{eq: Delta^2 z global bound} we have $|\Delta^2z\cdot z|\leq C\lambda^2\rho_z^2$ which gives, further using \eqref{eq: w and W equiv.} and Lemma \ref{lemma: L^ infintity and 2 bounds},
\begin{equation*}
    \varepsilon\int|\Delta W|^2-(\Delta^2z\cdot z)|W|^2\geq -C\varepsilon\lambda^2\|W\|_z^2\geq-\frac{c_0}{32}\|w\|^2_z
\end{equation*}
by taking $\varepsilon_0,\delta_0$ small enough. This completes the result.

\end{proof}

\section{Final results}

We start by presenting proofs of the main two theorems, these rely on many extra bounds from the final two sections.

\begin{proof}[Proof of Theorem \ref{thm: ratio of lambda to epsilon}]

Take $\varepsilon_0,\delta_0>0$ small, to be determined. Now take $u$ some degree $1$ $\varepsilon$-harmonic map with $0<\varepsilon<\varepsilon_0$ and $E_\varepsilon[u]\leq 4\pi+\delta_0$. Then by Lemma \ref{lemma: z is close!} there exists some $z\in\mathcal{Z}$ with $\|u-z\|_z=\inf_{\eta\in\mathcal{Z}}\{\|u-\eta\|_\eta\}$ by taking $\varepsilon_0,\delta_0$ small enough. Set $w=u-z$, $u_t=P(z+t(u-z))$ and $W_t=\partial_tu_t$, in particular set $W:=W_0=\textnormal{d}P(z)[w]$.
This gives
\begin{align*}
    0=\textnormal{d}E_\varepsilon(u)[W_1]=&\int_0^1\frac{\partial}{\partial t}\textnormal{d}E_\varepsilon(u_t)[W_t]+\textnormal{d}E_\varepsilon(z)[W]\\
    =&\int_0^1\textnormal{d}^2E_\varepsilon(u_t)[W_t,W_t]-\textnormal{d}^2E_\varepsilon(z)[W,W]+\textnormal{d}E_\varepsilon(u_t)[\textnormal{d}P(u_t)\partial_tW_t]
    \\&+\textnormal{d}^2E_\varepsilon(z)[W,W]+\textnormal{d}E_\varepsilon(z)[W].\\
\end{align*}
Using Lemma \ref{lemma: non-degeneracy} we have $\textnormal{d}^2E_\varepsilon(z)[W,W]\geq c_0\|w\|_z^2$.
Using Lemma \ref{Lemma: Ben's 5.2} as well as Lemma \ref{lemma: z is close!}, Lemma \ref{lemma: L^ infintity and 2 bounds} and \eqref{eq: Delta into z norm} we have
\begin{align*}
    |\textnormal{d}^2E_\varepsilon(u_t)[W_t,W_t]-\textnormal{d}^2E_\varepsilon(z)[W,W]|+|\textnormal{d}E_\varepsilon(u_t)[\textnormal{d}P(u_t)\partial_tW_t]|\\
    \leq C((\|w\|_{L^\infty}+\varepsilon\lambda^2+\varepsilon\|\nabla w\|_{L^{\infty}}^2)\|w\|_z^2+\varepsilon\|\Delta w\|_{L^2}^2)\\
    \leq \frac{c_0}{2}\|w\|^2_z+C\varepsilon\lambda^2\|w\|_z
\end{align*} for $\varepsilon_0,\delta_0$ small enough.
We also have using Lemma \ref{Lemma: Ben's 5.1} and \eqref{eq: w and W equiv.}
\begin{equation*}
    \textnormal{d}E_\varepsilon(z)[W]\leq C(\frac{(\log\lambda)^{1/2}}{\lambda^2}+\varepsilon\lambda^2)\|w\|_z.
\end{equation*}
Combining these gives
\begin{equation}\label{eq: w z norm bound}
    \|w\|_z\leq C(\frac{(\log\lambda)^{1/2}}{\lambda^2}+\varepsilon\lambda^2).
\end{equation}
For a general $T\in T_z\mathcal{Z}$ with $T_t=\textnormal{d}P(u_t)[T]$
\begin{align*}
    0=\textnormal{d}E_\varepsilon(u)[T_1]=&\int_0^1\frac{\partial}{\partial t}\textnormal{d}E_\varepsilon(u_t)[T_t]+\textnormal{d}E_\varepsilon(z)[T]\\
    =&\int_0^1\textnormal{d}^2E_\varepsilon(u_t)[T_t,W_t]-\textnormal{d}^2E_\varepsilon(z)[T,W]+\textnormal{d}E_\varepsilon(u_t)[\textnormal{d}P(u_t)\partial_tT_t]
    \\&+\textnormal{d}^2E_\varepsilon(z)[T,W]+\textnormal{d}E_\varepsilon(z)[T].\\
\end{align*}
Setting $T=\partial_\lambda z$ then using \eqref{eq: Delta into z norm}, \eqref{eq: w z norm bound}, Lemma \ref{Lemma: Ben's 5.1}, Lemma
\ref{Lemma: Ben's 5.3} and Lemma \ref{lemma: L^ infintity and 2 bounds} gives, for $\varepsilon_0,\delta_0$ small enough,
\begin{align*}
    |\textnormal{d}^2E_\varepsilon(u_t)[T_t,W_t]-\textnormal{d}^2E_\varepsilon(z)[T,W]|+|\textnormal{d}E_\varepsilon(u_t)[\textnormal{d}P(u_t)\partial_tT_t]|+|\textnormal{d}^2E_\varepsilon(z)[T,W]|\\
    \leq C\frac{1}{\lambda}(\|w\|_z^2+\varepsilon(\lambda^2+\|\nabla w\|^2_{L^2})\|w\|_z^2+\varepsilon\|\Delta w\|^2_{L^2})+C\frac{1}{\lambda}(\frac{(\log\lambda)^{1/2}}{\lambda^2}+\varepsilon\lambda^2)\|w\|_z\\
    \leq C(\frac{\log\lambda}{\lambda^5}+\varepsilon\frac{(\log\lambda)^{1/2}}{\lambda}+\varepsilon^2\lambda^3+\varepsilon^3\lambda^5).
\end{align*}
From Lemma \ref{lemma: partial lambda energy} we have
\begin{equation*}
    \textnormal{d}E_\varepsilon(z)[T]=\partial_\lambda E_\varepsilon[z]=8\pi\mathcal{J}(a)\frac{1}{\lambda^3}+\frac{64\pi}{3c_\gamma}\varepsilon\lambda+\mathcal{O}(\frac{1}{\lambda^4})+\mathcal{O}(\frac{\varepsilon}{\lambda}).
\end{equation*}
So combining we get
\begin{equation*}
    8\pi\mathcal{J}(a)\frac{1}{\lambda^3}+\frac{64\pi}{3c_\gamma}\varepsilon\lambda=\mathcal{O}(\frac{\log\lambda}{\lambda^5})+\mathcal{O}(\varepsilon\frac{(\log\lambda)^{1/2}}{\lambda})+\mathcal{O}(\varepsilon^2\lambda^3)+\mathcal{O}(\varepsilon^3\lambda^5).
\end{equation*}
We may assume $\varepsilon\lambda^2$ is as small as we like from Lemma \ref{lemma: L^ infintity and 2 bounds} which then gives the existence of some $C$ such that
\begin{equation}\label{eq: epsilon to lambda ratio}
    C^{-1}\frac{1}{\lambda^4}\leq\varepsilon\leq C\frac{1}{\lambda^4}
\end{equation}
and so we obtain
\begin{equation*}
    |\mathcal{J}(a)\frac{1}{\lambda^4}+\frac{8}{3c_\gamma}\varepsilon|=\mathcal{O}(\frac{\log\lambda}{\lambda^6})=\mathcal{O}(\varepsilon^{3/2}|\log\varepsilon|).
\end{equation*}
\end{proof}

\begin{proof}[Proof of Theorem \ref{thm: J must be critical}]

Under these same conditions now set $T=\nabla_A z$.
Using \eqref{eq: epsilon to lambda ratio} as well as \eqref{eq: Delta into z norm}, Lemma \ref{Lemma: Ben's 5.1}, Lemma \ref{Lemma: Ben's 5.3} and Lemma \ref{lemma: L^ infintity and 2 bounds} we get
\begin{align*}
    |\textnormal{d}^2E_\varepsilon(u_t)[T_t,W_t]-\textnormal{d}^2E_\varepsilon(z)[T,W]|+|\textnormal{d}E_\varepsilon(u_t)[\textnormal{d}P(u_t)\partial_tT_t]|+|\textnormal{d}^2E_\varepsilon(z)[T,W]|\\
    \leq C\lambda(\|w\|_z^2+\varepsilon(\lambda^2+\|\nabla w\|^2_{L^2})\|w\|_z^2+\varepsilon\|\Delta w\|^2_{L^2})+C\lambda(\frac{(\log\lambda)^{1/2}}{\lambda^2}+\varepsilon\lambda^2)\|w\|_z\\
    \leq C\frac{\log\lambda}{\lambda^3}
\end{align*}
for $\varepsilon_0,\delta_0$ small enough. Lemma \ref{lemma: nabla A energy} in combination with \eqref{eq: epsilon to lambda ratio} says that
\begin{align*}
    \nabla_A  E_\varepsilon[z]=&4\pi\nabla_A\mathcal{J}(a)\frac{1}{\lambda^2}+\mathcal{O}(\frac{1}{\lambda^3})+\mathcal{O}(\varepsilon\lambda)\\
    =&4\pi\nabla_A\mathcal{J}(a)\frac{1}{\lambda^2}+\mathcal{O}(\frac{1}{\lambda^3}).
\end{align*}
So combining these equations we get
\begin{equation*}
    \nabla_A\mathcal{J}(a)=\mathcal{O}(\frac{\log\lambda}{\lambda})=\mathcal{O}(\varepsilon^{1/4}|\log\varepsilon|).
\end{equation*}

\end{proof}

\section{Energy calculations}

In this section we find exact expansions for the derivatives of the energy. One can see that these align with what we would get by na\"ively differentiating the expansion in Lemma \ref{Lemma: Raw energy bound}.

\begin{lemma}\label{lemma: partial lambda energy}
     For $z\in \mathcal{Z}$ and $0<\varepsilon<1$ we have
     \begin{equation*}
        \partial_\lambda            E_\varepsilon[z]=8\pi\mathcal{J}(a)\frac{1}{\lambda^3}+\frac{64\pi}{3c_\gamma}\varepsilon\lambda+\mathcal{O}(\frac{1}{\lambda^4})+\mathcal{O}(\frac{\varepsilon}{\lambda}).
    \end{equation*}
\end{lemma}

\begin{proof}

We first note that
\begin{align*}
    \partial_\lambda E_\varepsilon[z]=&\int_\Sigma \nabla z\cdot\partial_\lambda\nabla z+\varepsilon\Delta z\cdot\partial_\lambda\Delta z\text{d}\Sigma\\
    =&\int_{\Sigma\setminus U_r}-\Delta z\cdot\partial_\lambda z+\varepsilon\Delta^2z\cdot\partial_\lambda z\text{d}\Sigma-\int_{\mathbb{D}_r}\Delta z\cdot\partial_\lambda z\text{d}x+\varepsilon\int_{\mathbb{D}_r}(\Delta^g)^2z\cdot\partial_\lambda z\text{d}x_g.
\end{align*}
From \eqref{eq: z bounds away from D_r} and \eqref{eq: lambda z bounds away from D_r} we have $|-\Delta z\cdot\partial_\lambda z+\varepsilon\Delta^2z\cdot\partial_\lambda z|=\mathcal{O}(\frac{1}{\lambda^4})$ on $\Sigma\setminus U_r$.
For the second term we note that $\Delta\pi_\lambda\cdot\partial_\lambda\pi_\lambda=0$, so using \eqref{eq: extra j and K bounds} we get
\begin{equation*}
    \Delta z\cdot\partial_\lambda z=\Delta\pi_\lambda\cdot \partial_\lambda j_\lambda^\top+\Delta j_\lambda^\top\cdot \partial_\lambda \pi_\lambda+\mathcal{O}(\frac{1}{\lambda^2}\frac{|x|}{1+\lambda^2|x|^2}).
\end{equation*}
By Taylor expanding we obtain
\begin{align*}
    \Delta\pi_\lambda\cdot \partial_\lambda j_\lambda^\top+\Delta j_\lambda^\top\cdot \partial_\lambda \pi_\lambda
    =&-\frac{1}{\lambda}(x^a\partial_aj\cdot\pi_\lambda)|\nabla\pi_\lambda|^2\\
    =& -\frac{32\lambda}{(1+\lambda^2|x|^2)^3}\bigg((x^1)^2\partial_{x^1y^1}J(0,0)+(x^2)^2\partial_{x^2y^2}J(0,0)\\
    &\qquad\qquad\quad+x^1x^2(\partial_{x^1y^2}J(0,0)+\partial_{x^2y^1}J(0,0))\bigg)\\
    &+\mathcal{O}(\frac{\lambda|x|^3}{(1+\lambda^2|x|^2)^3}).
\end{align*}
Now putting these estimates together along with the fact that we are integrating on a disc, we get
\begin{align}
\begin{split}\label{eq: 4.1 final 2}
  -\int_{\mathbb{D}_r}\Delta z\cdot\partial_\lambda z\text{d}x=&\int_{\mathbb{D}_r}\frac{16\lambda|x|^2}{(1+\lambda^2|x|^2)^3}\mathcal{J}(a)+\mathcal{O}(\frac{\lambda|x|^3}{(1+\lambda^2|x|^2)^3}+\frac{1}{\lambda^2}\frac{|x|}{1+\lambda^2|x|^2})\\
  %=&\frac{32\pi}{\lambda^3}\mathcal{J}(a)\big[-\frac{2s^2+1}{4(1+s^2)^2}\big]^{\lambda r}_0+\mathcal{O}(\frac{1}{\lambda^4})\\
  =&8\pi\mathcal{J}(a)\frac{1}{\lambda^3}+\mathcal{O}(\frac{1}{\lambda^4})
\end{split}
\end{align}
where $\mathcal{J}(a)=\partial_{x^1y^1}J_a(0,0)+\partial_{x^2y^2}J_a(0,0)$.
For the final term we use \eqref{eq: lambda z remainders} and \eqref{eq: Delta squared z metric bound} to get on $\mathbb{D}_r$
\begin{align*}
    (\Delta^g)^2z\cdot \partial_\lambda z=&\frac{1}{c_\gamma^2}\Delta^2\pi_\lambda\cdot\partial\pi_\lambda+\mathcal{O}(\frac{\lambda^2|x|}{(1+\lambda^2|x|^2)^{3}})\\
    =&\frac{256\lambda^5|x|^2}{(1+\lambda^2|x|^2)^5}+\mathcal{O}(\frac{\lambda^2|x|}{(1+\lambda^2|x|^2)^{3}}).
\end{align*}
One can calculate
\begin{align}
\begin{split}\label{eq: 4.1 final 3}
    \int_{\mathbb{D}_r}\frac{|x|^2}{(1+\lambda^2|x|^2)^5}\text{d}x_g
    %=&2\pi\frac{c_\gamma}{\lambda^4}\int_0^{\lambda r}\frac{s^3}{(1+s^2)^5}+\mathcal{O}(\frac{1}{\lambda^6})\\
    =&\frac{\pi}{12}\frac{c_\gamma}{\lambda^4}+\mathcal{O}(\frac{1}{\lambda^6})\\
    \int_{\mathbb{D}_r}\frac{\lambda^2|x|}{(1+\lambda^2|x|^2)^{3}}\text{d}x_g=&\mathcal{O}(\frac{1}{\lambda}).
\end{split}
\end{align}
Combining \eqref{eq: 4.1 final 2} and \eqref{eq: 4.1 final 3} completes the proof.
\end{proof}

\begin{lemma}\label{lemma: nabla A energy}
    For $z\in \mathcal{Z}$ and $0<\varepsilon<1$ we have
    \begin{equation*}
        \nabla_A  E_\varepsilon[z]=4\pi\nabla_A\mathcal{J}(a)\frac{1}{\lambda^2}+\mathcal{O}(\frac{1}{\lambda^3})+\mathcal{O}(\varepsilon\lambda).
    \end{equation*}
\end{lemma}

\begin{proof}

We first note that
\begin{align*}
    \nabla_A  E_\varepsilon[z]=&\int_\Sigma \nabla z\cdot\nabla_A \nabla z+\varepsilon\Delta z\cdot\nabla_A \Delta z\text{d}\Sigma\\
    =&\int_{\Sigma\setminus U_r}-\Delta z\cdot\nabla_A  z+\varepsilon\Delta^2 z\cdot\nabla_A  z\text{d}\Sigma-\int_{\mathbb{D}_r}\Delta z\cdot\nabla_A  z\text{d}x\\
    &+\varepsilon\int_{\mathbb{D}_r}(\Delta^g)^2z\cdot \nabla_A z\text{d}x_g.
\end{align*}
Using \eqref{eq: z bounds away from D_r} and \eqref{eq: nabla A z away from D_r} we have $|-\Delta z\cdot\nabla_A  z+\varepsilon\Delta^2 z\cdot\nabla_A  z|=\mathcal{O}(\frac{1}{\lambda^3})$ on $\Sigma\setminus U_r$.
We calculate, using \eqref{eq: extra j and K bounds} and \eqref{eq: sharper j and K bounds},
\begin{align*}
    \Delta z\cdot\nabla_A  z=&\Delta\pi_\lambda\cdot\nabla_Aj^\top_\lambda+\Delta j^\top_\lambda\cdot\nabla_A \pi_\lambda+\mathcal{O}(\frac{|x|}{(1+\lambda^2|x|^2)^2}+\frac{1}{\lambda^3})\\
    =&|\nabla\pi_\lambda|^2A^m\partial_m j_\lambda\cdot\nabla\pi_\lambda+\mathcal{O}(\frac{|x|}{(1+\lambda^2|x|^2)^2}+\frac{1}{\lambda^3}).
\end{align*}
By Taylor expanding we have
\begin{align*}
    A^m\partial_m j_\lambda\cdot\nabla\pi_\lambda=&\frac{4}{1+\lambda^2|x|^2}\bigg(A^m\partial_mj_\lambda(0)\cdot x\\
    &+(x^1)^2(A^1\partial^3_{x^1x^1y^1}J(0,0)+A^2\partial^3_{x^2x^1y^1}J(0,0))\\
    &+(x^2)^2(A^1\partial^3_{x^1x^2y^2}J(0,0)+A^2\partial^3_{x^2x^2y^2}J(0,0))+Cx^1x^2\bigg)\\
    &+\mathcal{O}(\frac{|x|^3}{1+\lambda^2|x|^2}).
\end{align*}
We recall the following lemma.
\begin{lemma}[Lemma B.1 \cite{sharp_lowenergyalpha}]
With $J_a$ and $\mathcal{J}$ as in the definitions of $\Sigma$ and $z$ we have the following.
    \begin{equation*}
    -\frac{1}{2}\nabla_A\mathcal{J}(a)=A^i\partial_{x^i}(\partial^2_{x^1y^1}J_a(0,0)+\partial^2_{x^2y^2}J_a(0,0)).
\end{equation*}
\end{lemma}
This gives us
\begin{align*}
    -\int_{\mathbb{D}_r}\Delta z\cdot\nabla_A  z\text{d}x=&8\lambda^2\nabla_A\mathcal{J}(a)\int_{\mathbb{D}_r}\frac{|x|^2}{(1+\lambda^2|x|^2)^3}+\mathcal{O}(\frac{1}{\lambda^3})\\
    =&4\pi\nabla_A\mathcal{J}(a)\frac{1}{\lambda^2}+\mathcal{O}(\frac{1}{\lambda^3}).
\end{align*}
For the third term we get, using \eqref{eq: z remainders} and \eqref{eq: Delta squared z metric bound},
\begin{align*}
    (\Delta^g)^2z\cdot \nabla_A z=\frac{1}{c_\gamma^2}\Delta^2\pi_\lambda\cdot(-A^j\nabla_j\pi_\lambda)+\mathcal{O}(\frac{\lambda^3}{(1+\lambda^2|x|^2)^2}).
\end{align*}
Then using \eqref{eq: exact pi derivatives} along with conformality of $\pi_\lambda$ we get
\begin{align*}
    \int_{\mathbb{D}_r}\frac{1}{c_\gamma^2}\Delta^2\pi_\lambda\cdot(-A^j\nabla_j\pi_\lambda) \text{d}x_g
    =&\int_{\mathbb{D}_r}-\frac{1}{c_\gamma^2}\frac{32\lambda^4}{(1+\lambda^2|x|^2)^3}|\nabla\pi_\lambda|^2(A\cdot x)\text{d}x_g=0,\\
    \int_{\mathbb{D}_r}\frac{\lambda^3}{(1+\lambda^2|x|^2)^3} \text{d}x_g=&\mathcal{O}(\lambda).
\end{align*}
Combining these results completes the proof.

\end{proof}

\section{Variation calculations}

We will now prove a number of bounds on the first and second variation of the $\varepsilon$-energy, which we use for the final bound. We shall be concerned with maps with values in $z^*T\mathbb{S}^2$ almost everywhere, so we define
\begin{equation*}
    \Gamma^2(z)=\{V\in W^{2,2}(\Sigma,\mathbb{R}^3) \text{ s.t. }V(x)\in T_{z(x)}\mathbb{S}^2 \text{ a.e.}\}.
\end{equation*}

\subsection{Estimates for $z$}

\begin{lemma}\label{Lemma: Ben's 5.1}
    Let $0\leq \varepsilon,\delta\leq 1$ and $z\in \mathcal{Z}$ with $E_\varepsilon[z]\leq\delta$ and $V\in \Gamma^2(z)$. Then there exists $C(\Sigma)$ such that
    \begin{equation*}
        |\textnormal{d}E_\varepsilon(z)[V]|\leq C(\frac{(\log\lambda)^{1/2}}{\lambda^2}+\varepsilon\lambda^2)\|V\|_z,
    \end{equation*}
    \begin{equation*}
        |\textnormal{d}^2E_\varepsilon[\partial_\lambda z,V]|\leq C\frac{1}{\lambda}(\frac{(\log\lambda)^{1/2}}{\lambda^2}+\varepsilon\lambda^2)\|V\|_z
    \end{equation*}
    and
    \begin{equation*}
        |\textnormal{d}^2E_\varepsilon(z)[\nabla_A z,V]|\leq C\lambda(\frac{(\log\lambda)^{1/2}}{\lambda^2}+\varepsilon\lambda^2)\|V\|_z.
    \end{equation*}
\end{lemma}

\begin{proof}

We first note that
\begin{align*}
    \text{d}E_\varepsilon(z)[V]=&\int_\Sigma -\Delta z\cdot V+\varepsilon\Delta^2 z\cdot V\text{d}\Sigma\\
    =&\int_{\Sigma\setminus U_r(a)} -\Delta z \cdot V\text{d}\Sigma-\int_{\mathbb{D}_r}\Delta z\cdot V\text{d}x+\varepsilon\int_{\Sigma}\Delta^2 z\cdot V\text{d}\Sigma
\end{align*}
For the first term we have that, using \eqref{eq: z bounds away from D_r} and Remark \ref{remark: log estimate},
\begin{equation}\label{eq: 5.1.1 final 1}
    |\int_{\Sigma\setminus U_r(a)} -\Delta z \cdot V|\leq C\frac{1}{\lambda^2}\int_{\Sigma}|V|\leq C\frac{(\log \lambda)^{1/2}}{\lambda^2}\|V\|_z.
\end{equation}
For the second term note that, as $V\in\Gamma^2(z)$, we only need to consider the part of $\Delta z$ tangential to $z$.
So using \eqref{eq: Delta z in the ball} we get
\begin{align*}
    (\Delta z)^{\top_z}=&(\Delta z)^{\top_\pi}+\mathcal{O}(\frac{\lambda|x|}{(1+\lambda^2|x|^2)^2})\\
    =&(\Delta \pi+\mathcal{O}(\frac{1}{(1+\lambda^2|x|^2)^{3/2}}))^{\top_\pi}+\mathcal{O}(\frac{\lambda|x|}{(1+\lambda^2|x|^2)^2})\\
    =&\mathcal{O}(\frac{1}{(1+\lambda^2|x|^2)^{3/2}}).
\end{align*}
Where we have used that $|y^{\top_z}-y^{\top_\pi}|\leq |j||y|\leq \frac{|x|}{\lambda}|y|$ for any $y\in\mathbb{R}^3$.
This gives
\begin{align}
\begin{split}\label{eq: 5.1.1 final 2}
    |\int_{\mathbb{D}_r}\Delta z\cdot V|\leq&C\int_{\mathbb{D}_r}\frac{1}{(1+\lambda^2|x|^2)^{3/2}}|V|\\
    \leq&C\bigg(\int_{\mathbb{D}_r}\frac{1}{\lambda^2(1+\lambda^2|x|^2)}\bigg)^{1/2}\bigg(\int_{\mathbb{D}_r}\rho_z^2|V|\bigg)^{1/2}\\
    \leq& C\frac{(\log \lambda)^{1/2}}{\lambda^2}\|V\|_z.
\end{split}
\end{align}
For the third term we use \eqref{eq: Delta^2 z global bound} to get
\begin{align}
\begin{split}\label{eq: 5.1.1 final 3}
    |\int_{\Sigma}\Delta^2 z\cdot V|\leq\int_\Sigma\lambda^2\rho_z^2|V|\leq&\bigg(\int_{\Sigma}\lambda^4\rho_z^2\bigg)^{1/2}\bigg(\int_{\Sigma}\rho_z^2 |V|^2\bigg)^{1/2}\\
    \leq& C\lambda^2\|V\|_z.
\end{split}
\end{align}
Combining \eqref{eq: 5.1.1 final 1}, \eqref{eq: 5.1.1 final 2} and \eqref{eq: 5.1.1 final 3} gives the first inequality.

\vspace{5mm}

For the second inequality we first note that we have, using Remark \ref{remark: log estimate}, \eqref{eq: variation expansions}, \eqref{eq: z bounds away from D_r} and \eqref{eq: lambda z bounds away from D_r},
\begin{align*}
    |\text{d}^2E_\varepsilon(z)[\partial_\lambda z,V]|=&|\int_\Sigma -(\partial_\lambda \Delta z+ |\nabla z|^2\partial_\lambda z)\cdot V+\varepsilon(\partial_\lambda\Delta^2 z-(\Delta^2 z\cdot z)\partial_\lambda z)\cdot V\text{d}\Sigma|\\
    \leq&C\frac{(\log \lambda)^{1/2}}{\lambda^3}\|V\|_z-\int_{\mathbb{D}_r}(\partial_\lambda \Delta z+|\nabla z|^2\partial_\lambda z)\cdot V\text{d}x\\
    &\quad\quad+\varepsilon\int_{\Sigma}(\partial_\lambda\Delta^2 z-(\Delta^2 z\cdot z)\partial_\lambda z)\cdot V\text{d}\Sigma.
\end{align*}
For the second term we note that from \eqref{eq: z remainders}, \eqref{eq: lambda z remainders} and \eqref{eq: Delta lambda z tangent bound} we get
\begin{align*}
    (\partial_\lambda\Delta z+
    |\nabla z|^2\partial_\lambda z)^{\top_z}=&(\partial_\lambda\Delta \pi_\lambda +|\nabla \pi_\lambda |^2\partial_\lambda\pi_\lambda)^{\top_z}  +\mathcal{O}(\frac{|x|}{\lambda^2}\rho_z^2).
\end{align*}
We can then calculate
\begin{equation*}
    (\partial_\lambda\Delta \pi_\lambda +|\nabla \pi_\lambda |^2\partial_\lambda\pi_\lambda )^{\top_z}=-\partial_\lambda(|\nabla \pi_\lambda |^2)\pi_\lambda^{\top_z}=\mathcal{O}(\frac{|x|}{\lambda^2}\rho_z^2).
\end{equation*}
Which then gives
\begin{equation}\label{eq: 5.1.2 final 2}
        \begin{split}
        \int_{\mathbb{D}_r}(\partial_\lambda \Delta z+|\nabla z|^2\partial_\lambda z)\cdot V\text{d}x
        \leq&\frac{1}{\lambda^2}\int_{\mathbb{D}_r}|x|\rho_z^2|V|\\
        \leq&\frac{1}{\lambda^2}(\int_{\mathbb{D}_r}|x|^2\rho_z^2)^{1/2}(\int_{\mathbb{D}_r}\rho_z^2|V|^2)^{1/2}
        \leq C\frac{(\log\lambda)^{1/2}}{\lambda^3}\|V\|_z.
        \end{split}
\end{equation}
For the third term we note that using \eqref{eq: lambda z remainders}, \eqref{eq: Delta^2 z global bound} and \eqref{eq: lambda delta^2 z global bound} we have globally
\begin{equation*}
    |\partial_\lambda(\Delta^g)^2 z-((\Delta^g)^2 z\cdot z)\partial_\lambda z|=\mathcal{O}(\lambda\rho_z^2)
\end{equation*}
giving
\begin{equation}\label{eq: 5.1.2 final 3}
    |\int_{\mathbb{D}_r}(\partial_\lambda(\Delta^g)^2 z-((\Delta^g)^2 z\cdot z)\partial_\lambda z)\cdot V\text{d}x_g|\leq C\lambda\int_{\mathbb{D}_r}\rho_z^2|V|
    \leq C \lambda\|V\|_z .
\end{equation}
Combining \eqref{eq: 5.1.2 final 2} and \eqref{eq: 5.1.2 final 3} gives the second inequality.

\vspace{5mm}

Now for the final inequality we have, using \eqref{eq: variation expansions}, \eqref{eq: z bounds away from D_r} and Remark \ref{remark: log estimate},
\begin{align*}
    |\text{d}^2E_\varepsilon(z)[\nabla_A z,V]|\leq&C \frac{(\log\lambda)^{1/2}}{\lambda}\|V\|_z+|\int_{\mathbb{D}_r}  \nabla \nabla_A z\cdot \nabla V-(|\nabla z|^2\nabla_A z)\cdot V\text{d}x|\\
    &+\varepsilon|\int_\Sigma(\Delta^2 \nabla_Az -(\Delta^2 z\cdot z) \nabla_Az)\cdot  V\text{d}\Sigma|.\\
\end{align*}
From \eqref{eq: z remainders} and \eqref{eq: nabla A z remainders} we have on $\mathbb{D}_r$
\begin{align*}
     \nabla \nabla_A z=& A^j\partial_j\nabla \pi_\lambda+\mathcal{O}(\frac{1}{(1+\lambda^2|x|^2)^{1/2}})\\
     |\nabla z|^2\nabla_A z =& |\nabla\pi_\lambda|^2A^j\partial_j \pi_\lambda+\mathcal{O}(\frac{\lambda}{(1+\lambda^2|x|^2)^2}).
\end{align*}
We calculate
\begin{equation*}
    \int_{\mathbb{D}_r}\frac{1}{(1+\lambda^2|x|^2)^{1/2}}|\nabla V|+\frac{\lambda}{(1+\lambda^2|x|^2)^2}|V|\leq C\frac{(\log\lambda)^{1/2}}{\lambda}\|V\|_z.
\end{equation*}
We also find that
\begin{equation*}
    (A^j\partial_j\Delta \pi_\lambda+ |\nabla \pi_\lambda|^2A^j\partial_j \pi_\lambda)^{\top_z}=-(A^j\partial_j|\nabla \pi_\lambda|^2)(\pi_\lambda)^{\top_z}=\mathcal{O}(\frac{|x|^2\lambda^3}{(1+\lambda^2|x|^2)^{3}})
\end{equation*}
which gives, further using Remark \ref{remark: log estimate},
\begin{align}
\begin{split}\label{eq: 5.1.3 final 2}
    |\int_{\mathbb{D}_r}  A^j\partial_j\nabla  \pi_\lambda\cdot \nabla V-(|\nabla \pi_\lambda|^2A^j\partial_j \pi_\lambda)\cdot V|\leq&\int_{\partial\mathbb{D}_r}|A^j\partial_j\nabla \pi_\lambda|\cdot|V|\\&+ |\int_{\mathbb{D}_r}(A^j\partial_j\Delta \pi_\lambda+ |\nabla \pi_\lambda|^2A^j\partial_j \pi_\lambda)\cdot V|\\
    \leq &C\frac{1}{\lambda}\int_{\partial\mathbb{D}_r}|V|+C\int_{\mathbb{D}_r}\frac{|x|^2\lambda^2}{(1+\lambda^2|x|^2)^{2}}\rho_z|V|\\
    \leq & \frac{(\log\lambda)^{1/2}}{\lambda}\|V\|_z.
\end{split}
\end{align}
We also have from \eqref{eq: Delta^2 z global bound} and \eqref{eq: delta^2 nabla A z global bound} the estimate that globally
\begin{equation*}
    \Delta^2\nabla_Az+(\Delta^2 z\cdot z) \nabla_A z=\mathcal{O}(\lambda^3\rho_z^2).
\end{equation*}
Which gives
\begin{equation}\label{eq: 5.1.3 final 3}
    |\int_\Sigma (\Delta^2\nabla_Az+(\Delta^2 z\cdot z) \nabla_A z)\cdot V|\leq C\lambda^3\int_\Sigma\rho_z^2|V|\leq C\lambda^3\|V\|_z.
\end{equation}
Combining \eqref{eq: 5.1.3 final 2} and \eqref{eq: 5.1.3 final 3} completes the proof.

\end{proof}

\subsection{Difference estimates}

\begin{lemma}\label{Lemma: Ben's 5.2}
    Given $u\in W^{2,\infty}(\Sigma,\mathbb{S}^2)$ and $z\in\mathcal{Z}$ with $\|u-z\|_{L^\infty}\leq \frac{1}{2}$. Set $w=u-z$, $u_t=P(z+tw)$, $W_t=\textnormal{d}P(z+tw)[w]$ and $W=W_0$. Then there exists some constant $C=C(\Sigma)$ such that
    \begin{equation*}
        |\textnormal{d}^2E_\varepsilon(z)[W,W]-\textnormal{d}^2E_\varepsilon(u_t)[W_t,W_t]|\leq C(\|w\|_{L^\infty}\|w\|_z^2+\varepsilon(\lambda^2+\|\nabla w\|^2_{L^{\infty}})\|w\|_z^2+\varepsilon\|\Delta w\|_{L^2}^2)
    \end{equation*}
    and
    \begin{equation*}
        |\int_\Sigma \textnormal{d}E_\varepsilon(u_t)[\textnormal{d}P(u_t)[\partial_tW_t]]|\leq  C(\|w\|_{L^\infty}\|w\|_z^2+\varepsilon(\lambda^2+\|\nabla w\|^2_{L^\infty})\|w\|_z^2+\varepsilon\|\Delta w\|_{L^2}^2).
    \end{equation*}
\end{lemma}
\begin{proof}

We can calculate $w^{\top_z}=(w\cdot z)z=-\frac{1}{2}|w|^2z$ and $W=w+\frac{1}{2}|w|^2z$.
We then have
\begin{equation*}
    \nabla W=\nabla w+(w\cdot\nabla w)z+\frac{1}{2}|w|^2\nabla z.
\end{equation*}
This gives, using the uniform $L^\infty$ bound on $w$, that
\begin{comment}
\begin{align}
\label{eq: w and W equiv.}
    \frac{3}{4}\|w\|_{L^\infty}\leq\|W\|_{L^\infty}\leq \frac{5}{4}\|w\|_{L^\infty},
    \quad\quad \frac{1}{32}\|w\|_z^2\leq\|W\|_z^2\leq \frac{9}{2}\|w\|_z^2.
\end{align}
\end{comment}
\begin{align}
\label{eq: w and W equiv.}
    C_1\|w\|_{L^\infty}\leq\|W\|_{L^\infty}\leq C_2\|w\|_{L^\infty},
    \quad\quad C_3\|w\|_z^2\leq\|W\|_z^2\leq C_4\|w\|_z^2.
\end{align}
We then have the bounds on $W_t$ from the definitions, using that $|\nabla z|\leq C\rho_z$ from \eqref{eq: z remainders} and \eqref{eq: z bounds away from D_r},
\begin{equation}\label{eq: first W_t bounds}
\begin{split}
    |W_t|\leq& C|w|,\\
    |\nabla W_t|\leq& C(|\nabla w|+|w|\rho_z),\\
\end{split}\hspace{10mm}
\begin{split}
    |W_t-W|\leq& C|w|^2,\\
    |\nabla (W_t- W)|\leq& C(|w|\cdot|\nabla w|+|w|^2\rho_z).
\end{split}
\end{equation}
Also note that
\begin{align*}
    \Delta W_t=& dP(z+tw)[\Delta w]+2d^2P(z+tw)[\nabla w,\nabla z+t\nabla w]\\
    &+d^2P(z+tw)[w,\Delta z+t\Delta w]
    +d^3P(z+tw)[w,\nabla z+t\nabla w,\nabla z+t\nabla w].
\end{align*}
Which gives us the bounds, also using \eqref{eq: Delta z in the ball} and \eqref{eq: z bounds away from D_r},
\begin{align}
\begin{split}\label{eq: Delta W_t bounds}
    |\Delta W_t|\leq& C(|\Delta w|+|\nabla w|\rho_z+|\nabla w|^2+|w|\rho_z^2),\\
    |\Delta (W_t-W)|\leq & C(|w|\cdot|\Delta w|+|w|\cdot|\nabla w|\rho_z+|\nabla w|^2+|w|^2\rho_z^2).
\end{split}
\end{align}
We also note that
\begin{equation*}
    \Delta u_t=\text{d}P(z+tw)[\Delta z+t\Delta w]+\text{d}P^2(z+tw)[\nabla z+t\nabla w,\nabla z+t\nabla w]
\end{equation*}
giving the bounds on $u_t$, using the fact that $z=u_0$,
\begin{equation}\label{eq: u_t bounds}
\begin{split}
    |u_t|\leq& C\\
    |\nabla u_t|\leq& C(|\nabla w|+\rho_z)\\
    |\Delta u_t|\leq &C(|\nabla w|^2+|\nabla w|\rho_z\\&+\rho_z^2+|\Delta w|)\\
\end{split}
\hspace{10mm}
\begin{split}
    |u_t-z|\leq& C|w|\\
    |\nabla (u_t-z)|\leq& C(|\nabla w|+|w|\rho_z)\\
    |\Delta (u_t-z)|\leq& C(|\nabla w|^2+|\nabla w|\rho_z\\&+|w|\rho_z^2+|\Delta w|).\\
\end{split}
\end{equation}
We can then calculate, using \eqref{eq: variation expansions},
\begin{align*}
    &\text{d}^2E_\varepsilon(z)[W,W]-\text{d}^2E_\varepsilon(u_t)[W_t,W_t]=\\ &\int_\Sigma \nabla (W-W_t)\cdot\nabla(W+W_t)
    -|\nabla z|^2(|W|^2-|W_t|^2)
    -(|\nabla z|^2-|\nabla u_t|^2)|W_t|^2\\&+\varepsilon\Delta (W-W_t)\cdot\Delta(W+W_t)
    -\varepsilon\Delta^2z\cdot(z|W|^2-u_t|W_t|^2)-\varepsilon\Delta(z-u_t)\cdot\Delta(u_t|W_t|^2).
\end{align*}
So then using \eqref{eq: first W_t bounds}, \eqref{eq: Delta W_t bounds} and \eqref{eq: u_t bounds} and working through the terms we get
\begin{align*}
    |\int_\Sigma \nabla (W-W_t)\cdot\nabla(W+W_t)|
    \leq& C\int_\Sigma (|w|\cdot|\nabla w|+|w|^2\rho_z)(|\nabla w|+|w|\rho_z)\\
    \leq& C\|w\|_{L^\infty}\|w\|_z^2,\\
    |\int_\Sigma |\nabla z|^2(|W|^2-|W_t|^2)|=&|\int_\Sigma |\nabla z|^2 (W-W_t)\cdot(W+W_t)|\\
    \leq&C\int_\Sigma |w|^3\rho_z^2\\
    \leq&  C\|w\|_{L^\infty}\|w\|_z^2,\\
    |\int_\Sigma(|\nabla z|^2-|\nabla u_t|^2)|W_t|^2|=&
    |\int_\Sigma(\nabla z-\nabla u_t)\cdot(\nabla z+\nabla u_t)|W_t|^2|\\
    \leq&C\int_\Sigma(|\nabla w|+|w|\rho_z)(|\nabla w|+\rho_z)|w|^2\\
    \leq&  C\|w\|_{L^\infty}\|w\|_z^2,\\
    |\int_\Sigma \Delta (W-W_t)\cdot\Delta (W+W_t)|\leq& C\int_\Sigma(|\Delta w|+|\nabla w|\rho_z+|\nabla w|^2+|w|\rho_z^2)\\
    &\quad\cdot(|w|\cdot|\Delta w|+|w|\cdot|\nabla w|\rho_z+|\nabla w|^2+|w|^2\rho_z^2)\\
    \leq&C((\lambda^2+\|\nabla w\|^2_{L^\infty})\|w\|_z^2+\|\Delta w\|_{L^2}^2).
\end{align*}
For the next term note that
\begin{equation*}
    z|W|^2-u_t|W_t|^2=(z-u_t)|W|^2+u_t(W-W_t)\cdot (W+W_t)
\end{equation*}
so using \eqref{eq: first W_t bounds}, \eqref{eq: u_t bounds} and \eqref{eq: Delta^2 z global bound}
\begin{equation*}
   | \int_\Sigma \Delta^2z\cdot(z|W|^2-u_t|W_t|^2)|\leq C\int_\Sigma\lambda^2\rho_z^2|w|^3\leq C\lambda^2\|w\|^2_z.
\end{equation*}
For the final term we calculate, using \eqref{eq: first W_t bounds}, \eqref{eq: Delta W_t bounds} and \eqref{eq: u_t bounds}, that
\begin{align*}
    \Delta(u_t|W_t|^2)=&(\Delta u_t)|W_t|^2+4\nabla u_t\cdot\nabla W_t\cdot W_t+2u_t(\Delta W_t\cdot W_t+|\nabla W_t|^2)\\
    \leq& C(|w|^2\rho_z^2+|w|\cdot|\nabla w|\rho_z+|\nabla w|^2+|w|\cdot|\Delta w|).
\end{align*}
So using \eqref{eq: u_t bounds}
\begin{align*}
    &|\int_\Sigma \Delta(z-u_t)\cdot\Delta(u_t|W_t|^2)|\\
    &\leq C\int_\Sigma(|\nabla w|^2+|\nabla w|\rho_z+|w|\rho_z^2+|\Delta w|)\\
    &\quad\quad\cdot(|w|^2\rho_z^2+|w|\cdot|\nabla w|\rho_z+|\nabla w|^2+|w|\cdot|\Delta w|)\\
    &\leq C((\lambda^2+\|\nabla w\|^2_{L^\infty})\|w\|_z^2+\|\Delta w\|^2_{L^2}).
\end{align*}
Combining these estimates gives the first equation.

\vspace{5mm}

For the second equation we first note that
\begin{align*}
    W_t=\partial_tu_t=&\partial_t\frac{z+tw}{|z+tw|}\\
    =&\frac{w}{|z+tw|}-\frac{w\cdot(z+tw)}{|z+tw|^3}(z+tw)\\
    =&\frac{1}{|z+tw|}\text{d}P(u_t)[w].
\end{align*}
We then find that
\begin{align*}
    \text{d}P(u_t)[\partial_t W_t]
    %=&\text{d}P(u_t)[-2\frac{w\cdot(z+tw)}{|z+tw|^3}w+(-\frac{|w|^2}{|z+tw|^3}+3\frac{|w\cdot(z+tw)|^2}{|z+tw|^5})(z+tw)]\\
    =&-2\frac{1}{|z+tw|}(w\cdot u_t)W_t\\
    =&\frac{1}{|z+tw|}(|w|^2-2w\cdot (u_t-z))W_t
\end{align*}
using $w\cdot(w+2z)=0$.
Now directly differentiating and using \eqref{eq: w and W equiv.}, \eqref{eq: first W_t bounds}, \eqref{eq: Delta W_t bounds} and \eqref{eq: u_t bounds} gives
\begin{align}
\begin{split}\label{eq: dP of partial_t W_t bounds}
    |\text{d}P(u_t)[\partial_tW_t]|\leq& C |w|^3,\\
    |\nabla\text{d}P(u_t)[\partial_tW_t]|\leq&C(|w|^2|\nabla w|+|w|^3\rho_z),\\
    |\Delta\text{d}P(u_t)[\partial_tW_t]|\leq&C(|w|^2|\Delta w|+|w|\cdot|\nabla w|^2+|w|^2|\nabla w|\rho_z+|w|^3\rho_z^2).
\end{split}
\end{align}
This then gives, using \eqref{eq: variation expansions}, \eqref{eq: u_t bounds} and \eqref{eq: dP of partial_t W_t bounds},
\begin{align*}
    |\int_\Sigma \text{d}E_\varepsilon(u_t)[\text{d}P(u_t)[\partial_tW_t]]|\leq &C\int_\Sigma(|\nabla w|+\rho_z)(|w|^2|\nabla w|+|w|^3\rho_z)\\
    &+C\varepsilon\int_\Sigma(|\nabla w|^2+|\nabla w|\rho_z+\rho_z^2+|\Delta w|)\\
    &\quad\quad\cdot(|w|^2|\Delta w|+|w|\cdot|\nabla w|^2+|w|^2|\nabla w|\rho_z+|w|^3\rho_z^2)\\
    \leq&C(\|w\|_{L^\infty}\|w\|_z^2+\varepsilon(\lambda^2+\|\nabla w\|^2_{L^\infty})\|w\|_z^2+\varepsilon\|\Delta w\|_{L^2}^2)
\end{align*}
completing the proof.

\end{proof}

\begin{lemma}\label{Lemma: Ben's 5.3}
    Given $u\in W^{2,\infty}(\Sigma,\mathbb{S}^2)$ and $z\in\mathcal{Z}$ with $\|u-z\|_{L^\infty}\leq \frac{1}{2}$. Set $w=u-z$, $u_t=P(z+tw)$, $W_t=\textnormal{d}P(z+tw)[w]$ and $W=W_0$. Also given $T\in\Gamma^2(z)$ set $T_t=\textnormal{d}P(u_t)[T]$. Then there exists some constant  $C=C(\Sigma)$ such that
    \begin{enumerate}
        \item If $T=\partial_\lambda z$ then,
        \begin{equation*}
            |\textnormal{d}^2E_\varepsilon(z)[T,W]-\textnormal{d}^2E_\varepsilon(z)[T_t,W_t]|\leq C\frac{1}{\lambda}(\|w\|_z^2+\varepsilon(\lambda^2+\|\nabla w\|^2_{L^\infty})\|w\|_z^2+\varepsilon\|\Delta w\|_{L^2}^2)
        \end{equation*}
        and
        \begin{equation*}
            |\textnormal{d}E_\varepsilon(u_t)[\textnormal{d}P(u_t)[\partial_t T_t^\lambda]]|\leq C\frac{1}{\lambda}(\|w\|_z^2+\varepsilon(\lambda^2+\|\nabla w\|^2_{L^\infty})\|w\|_z^2+\varepsilon\|\Delta w\|_{L^2}^2).
        \end{equation*}
        \item If $T=\nabla_Az$ then,
        \begin{equation*}
            |\textnormal{d}^2E_\varepsilon(z)[T,W]-\textnormal{d}^2E_\varepsilon(z)[T_t,W_t]|\leq C\lambda(\|w\|_z^2+\varepsilon(\lambda^2+\|\nabla w\|^2_{L^\infty})\|w\|_z^2+\varepsilon\|\Delta w\|_{L^2}^2)
        \end{equation*}
        and
        \begin{equation*}
            |\textnormal{d}E_\varepsilon(u_t)[\textnormal{d}P(u_t)[\partial_t T_t^A]]|\leq C\lambda(\|w\|_z^2+\varepsilon(\lambda^2+\|\nabla w\|^2_{L^\infty})\|w\|_z^2+\varepsilon\|\Delta w\|_{L^2}^2).
        \end{equation*}
    \end{enumerate}
    
\end{lemma}

\begin{proof}

Note that we have, using \eqref{eq: variation expansions},
\begin{align}
\begin{split}\label{eq: second variation T expansion}
    \text{d}^2E_\varepsilon(z)[T,W]-\text{d}^2E_\varepsilon(z)[T_t,W_t]=&\int_\Sigma\nabla W\cdot\nabla(T-T_t)+\nabla(W-W_t)\cdot\nabla T_t\\
    &\quad-(|\nabla z|^2-|\nabla u_t|^2)(T\cdot W)-|\nabla u_t|^2(T\cdot W-T_t\cdot W_t)\\
    &+\varepsilon\int_\Sigma\Delta W\cdot\Delta(T-T_t)+\Delta(W-W_t)\cdot\Delta T_t\\
    &\quad-\Delta^2z\cdot(z(T\cdot W)-u_t(T_t\cdot W_t))\\
    &\quad-\Delta(z-u_t)\cdot\Delta(u_t(T_t\cdot W_t)).
\end{split}
\end{align}
Also note that
\begin{align*}
    \nabla T_t=&\text{d}P(u_t)[\nabla T]+\text{d}^2P(u_t)[T,\nabla u_t],\\
    \Delta T_t=&\text{d}P(u_t)[\Delta T]+\text{d}^2P(u_t)(2[\nabla T,\nabla u_t]+[T,\Delta u_t])+\text{d}^2P(u_t)[T,\nabla u_t,\nabla u_t].
\end{align*}
This gives, using \eqref{eq: u_t bounds},
\begin{equation}
\begin{split}\label{eq: Initial T bounds}
    |T_t|\leq& |T|,\\
    |\nabla T_t|\leq &C\big(|T|(\rho_z+|\nabla w |)+|\nabla T|\big),\\
    |\Delta T_t|\leq &C\big(|T|(\rho_z^2+|\nabla w|\rho_z+|\nabla w|^2+|\Delta w|)\\
    &\quad\quad+|\nabla T|(\rho_z+|\nabla w |)+|\Delta T|\big),\\
    |T_t-T|\leq& C|T|\cdot|w|,\\
    |\nabla(T_t-T)|\leq& C\big(|T|(|w|\rho_z+|\nabla w |)+|\nabla T|\cdot|w|\big),\\
    |\Delta(T_t-T)|\leq& C\big(|T|(|w|\rho_z^2+|\nabla w|\rho_z+|\nabla w|^2+|\Delta w|)\\
    &\quad\quad+|\nabla T|(|w|\rho_z+|\nabla w |)+|\Delta T|\cdot|w|\big).
\end{split}
\end{equation}
In the case of $T=\partial_\lambda z$ we have the bounds from \eqref{eq: lambda z remainders} and \eqref{eq: lambda z bounds away from D_r},
\begin{align*}
    |T|\leq& C\frac{1}{\lambda}\\
    |\nabla T|\leq&C\frac{1}{\lambda}\rho_z\\
    |\Delta T|\leq&C \frac{1}{\lambda}\rho_z^2
\end{align*}
this then gives the bounds using \eqref{eq: Initial T bounds}
\begin{align}
\begin{split}\label{eq: T^lambda bounds}
    |T_t|\leq&C\frac{1}{\lambda},\\
    |\nabla T_t|\leq &C\frac{1}{\lambda}(\rho_z+|\nabla w |),\\
    |\Delta T_t|\leq &C\frac{1}{\lambda}(\rho_z^2+\rho_z|\nabla w|\\&+|\nabla w|^2+|\Delta w|),
\end{split}
\hspace{7mm}
\begin{split}
    |T_t-T|\leq& C\frac{1}{\lambda}|w|,\\
    |\nabla(T_t-T)|\leq& C\frac{1}{\lambda}(\rho_z|w|+|\nabla w |),\\
    |\Delta(T_t-T)|\leq& C\frac{1}{\lambda}(\rho_z^2|w|+\rho_z|\nabla w|\\&+|\nabla w|^2+|\Delta w|).
\end{split}
\end{align}
\begin{comment}
We can then obtain the bounds using \eqref{eq: first W_t bounds}, \eqref{eq: Delta W_t bounds}, \eqref{eq: u_t bounds} and \eqref{eq: T^lambda bounds}
\begin{align*}
    ||\nabla z|^2-|\nabla u_t|^2|\leq &C(\rho_z^2|w|+\rho_z|\nabla w|+|\nabla w|^2)\\
    |T\cdot W-T_t\cdot W_t|\leq & C\frac{1}{\lambda}|w|^2\\
    |z(T\cdot W)-u_t(T_t\cdot W_t)|\leq & C\frac{1}{\lambda}|w|^2\\
    \Delta (u_t(T_t\cdot W_t))\leq & C\frac{1}{\lambda}(|\Delta w|+|\nabla w|\rho_z+|\nabla w|^2+|w|\rho_z^2).
\end{align*}
\end{comment}
Inserting these bounds into \eqref{eq: second variation T expansion} as well as using bounds from \eqref{eq: first W_t bounds}, \eqref{eq: Delta W_t bounds}, \eqref{eq: u_t bounds}, \eqref{eq: T^lambda bounds} and \eqref{eq: Delta^2 z global bound} one can obtain the first inequality.

\vspace{5mm}

To get the second inequality we note that for $T\in \Gamma^2(z)$ we have
\begin{align*}
    \text{d}P(u_t)[\partial_t T_t]
    %=& \text{d}P(u_t)[\partial_t( T-(T\cdot u_t)u_t)]\\
    =&\text{d}P(u_t)[-(T\cdot W_t)u_t-(T\cdot u_t)W_t]\\
    =&-(T\cdot u_t)W_t\\
    =&(T\cdot(z-u_t))W_t
\end{align*}
using $T\cdot z=0$. 
By differentiating this form and appealing to \eqref{eq: first W_t bounds}, \eqref{eq: Delta W_t bounds}, \eqref{eq: u_t bounds} and \eqref{eq: T^lambda bounds} we get the bounds
\begin{align*}
    |\text{d}P(u_t)[\partial_t T_t]|\leq&C|T|\cdot|w|^2\\
    |\nabla \text{d}P(u_t)[\partial_t T_t]|\leq&C\big(|T|(|w|^2\rho_z+|w|\cdot|\nabla w|)+|\nabla T|\cdot|w|^2\big)\\
    |\Delta \text{d}P(u_t)[\partial_t T_t]|\leq & C\big(|T|(|\nabla w|^2+|w|\cdot|\nabla w|\rho_z+|w|^2\rho_z^2+|w|\cdot|\Delta w|)\\
    &+|\nabla T|(|w|^2\rho_z+|w|\cdot|\nabla w|)+|\Delta T|\cdot|w|^2\big).
\end{align*}
In the case of $T=\partial_\lambda z$, this gives
\begin{align*}
    |\text{d}P(u_t)[\partial_t T_t]|\leq&C\frac{1}{\lambda}|w|^2,\\
    |\nabla \text{d}P(u_t)[\partial_t T_t]|\leq&C\frac{1}{\lambda}(\rho_z|w|^2+|w|\cdot|\nabla w|),\\
    |\Delta \text{d}P(u_t)[\partial_t T_t]|\leq & C\frac{1}{\lambda}(|\nabla w|^2+\rho_z|w|\cdot|\nabla w|+\rho_z^2|w|^2+|w|\cdot|\Delta w|).
\end{align*}
From the above, \eqref{eq: variation expansions} and \eqref{eq: u_t bounds} we obtain
\begin{align*}
    \text{d}E_\varepsilon(u_t)[\text{d}P(u_t)[\partial_t T_t^\lambda]]\leq & \int_\Sigma \nabla u_t\cdot\nabla \text{d}P(u_t)[\partial_t T_t]+\varepsilon \Delta u_t\cdot\Delta \text{d}P(u_t)[\partial_t T_t]\\
    \leq&C\frac{1}{\lambda}(\|w\|_z^2+\varepsilon(\lambda^2+\|\nabla w\|^2_{L^\infty})\|w\|_z^2+\varepsilon\|\Delta w\|_{L^2}^2)
\end{align*}
giving the second inequality.

\vspace{5mm}

In the case of $T=\nabla_A z$ we instead have the bounds from \eqref{eq: nabla A z remainders} and \eqref{eq: nabla A z away from D_r}
\begin{align*}
    |T|\leq&C\lambda\\
    |\nabla T|\leq&C{\lambda}\rho_z\\
    |\Delta T|\leq & C\lambda\rho_z^2
\end{align*}
so the proof for the final two inequalities is identical.

\end{proof}

\appendix

\section{Initial bounds}\label{section: bounds on z}

In general we know that close enough to the bubble point $z$ will look like $\pi_\lambda$ plus some lower order remainder term for $\lambda$ large enough.
We would like to show this quantitatively.
We start by noting the following exact derivatives of $\pi_\lambda$
\begin{align}
\begin{split}\label{eq: exact pi derivatives}
    \pi_\lambda=&(\frac{2\lambda x}{1+\lambda^2|x|^2},\frac{1-\lambda^2|x|^2}{1+\lambda^2|x|^2}),\\
    \nabla\pi_\lambda=&\frac{2\lambda      }{(1+\lambda^2|x|^2)^2}
    \begin{pmatrix}
        1-\lambda^2x_1^2+\lambda^2x_2^2 & -2\lambda^2x_1x_2 & -2\lambda x_1\\
        -2\lambda^2x_1x_2    &  1+\lambda^2x_1^2-\lambda^2x_2^2  & -2\lambda x_2
    \end{pmatrix},\\
    \Delta\pi_\lambda=&\frac{-8\lambda^2}{(1+\lambda^2|x|^2)^2}\pi_\lambda,\\
    \nabla_i\Delta\pi_\lambda=&\frac{32\lambda^4x_i}{(1+\lambda^2|x|^2)^3}\pi_\lambda-\frac{8\lambda^2}{(1+\lambda^2|x|^2)^2}\nabla_i\pi_\lambda,\\
    \Delta^2\pi_\lambda=&\frac{96\lambda^4-160\lambda^6|x|^2}{(1+\lambda^2|x|^2)^4}\pi_\lambda+\frac{64\lambda^4x_i}{(1+\lambda^2|x|^2)^3}\nabla_i\pi_\lambda.\\
\end{split}
\end{align}
It is easy to see that for any $m\geq 1$ and $x\in\mathbb{D}_r$ we have the following bounds
\begin{align}
\begin{split}\label{eq: naive pi derivative bounds}
    |\nabla^m\pi_\lambda|\leq& C(m)\frac{\lambda^m}{(1+\lambda^2|x|^2)^{(m+1)/2}},\\
    |\nabla^m\Delta\pi_\lambda|\leq &C(m)\frac{\lambda^{m+2}}{(1+\lambda^2|x|^2)^{m/2+2}}.
\end{split}
\end{align}
We also have in $\mathbb{D}_r$ for $m\geq 1$
\begin{align}
\begin{split}\label{eq: naive j derivative bounds}
    |j_\lambda|\leq &C\frac{|x|}{\lambda},\\
    |\nabla^mj_\lambda|\leq& C(m)\frac{1}{\lambda}.
\end{split}    
\end{align}
In $\mathbb{D}_r$ we have $z=P(\pi_\lambda+j_\lambda)$. We can differentiate this equation to get
\begin{align*}
    \nabla z=& \text{d}P(\pi_\lambda+j_\lambda)[\nabla\pi_\lambda+\nabla j_\lambda]\\
    =&\nabla\pi_\lambda+\text{d}P(\pi_\lambda+j_\lambda)[\nabla j_\lambda]+(\text{d}P(\pi_\lambda+j_\lambda)-\text{d}P(\pi_\lambda))[\nabla\pi_\lambda]
\end{align*}
and in general we see that $\nabla^kz=\nabla^k\pi_\lambda+R_k$ where $R_k$ is a remainder consisting of terms of the following forms
\begin{enumerate}
    \item $(\text{d}^tP(\pi_\lambda+j_\lambda)-\text{d}^tP(\pi_\lambda))[\nabla^{a_1}\pi_\lambda,\dots,\nabla^{a_t}\pi_\lambda]$ where all the $a_i$ are strictly positive integers and $a_1+...+a_t=k$\\
    \item $\text{d}^tP(\pi_\lambda+j_\lambda)[\nabla^{b_1}j_\lambda,\dots,\nabla^{b_u}j_\lambda,\nabla^{c_1}\pi_\lambda,\dots,\nabla^{c_v}\pi_\lambda]$ where $u$, the $b_i$ and the $c_i$ are strictly positive integers, $v$ is a non negative integer, $u+v=t$ and $\Sigma_ib_i+\Sigma_jc_j=k$.
\end{enumerate}

Applying \eqref{eq: naive pi derivative bounds} and \eqref{eq: naive j derivative bounds} to these remainders we obtain the bounds in $\mathbb{D}_r$
\begin{align}
\begin{split}\label{eq: z remainders}
    z=&\pi_\lambda+\mathcal{O}(\frac{|x|}{\lambda}),\\
    %|R_1|=&\mathcal{O}(\frac{1}{\lambda})\\
    \nabla^k z=&\nabla^k\pi+\mathcal{O}(\frac{\lambda^{k-2}}{(1+\lambda^2|x|^2)^{(k-1)/2}})\quad\text{ for }k\geq 1.
\end{split}
\end{align}
For Laplacian terms we can get a slightly stronger bound. This uses the fact that $\Delta j_\lambda=0$. We obtain in $\mathbb{D}_r$ 
\begin{equation}\label{eq: Delta z in the ball}
    \nabla^k\Delta z=\nabla^k\Delta \pi_\lambda+\mathcal{O}(\frac{\lambda^{k}}{(1+\lambda^2|x|^2)^{k/2+1}})\quad\text{ for }k\geq 0.
\end{equation}

Using the nature of $z$ on $\mathbb{D}_{2r}\setminus\mathbb{D}_r$ and on $\Sigma\setminus U_{2r}(a)$, it is easy to see that
\begin{align}
\begin{split}\label{eq: z bounds away from D_r}
    |z|=&\mathcal{O}(1),\\
    |\nabla^k z|=&\mathcal{O}(\frac{1}{\lambda})\quad\text{ for }k\geq 1,\\
    |\nabla^k\Delta z|=&\mathcal{O}(\frac{1}{\lambda^2})\quad\text{ for }k\geq 0
\end{split}
\end{align}
on $\Sigma\setminus U_{r}(a)$.

\subsection{$\lambda$ derivatives}
We first note that
\begin{align*}
    \partial_\lambda\pi_\lambda= &\frac{1}{\lambda}(x\cdot\nabla)\pi_\lambda,\\
    \partial_\lambda j_\lambda =& -\frac{1}{\lambda}j.
\end{align*}
These gives us the following bounds for $m\geq1$
\begin{align*}
    |\nabla^m\partial_\lambda\pi_\lambda|\leq& C(m)\frac{\lambda^{m-1}}{(1+\lambda^2|x|^2)^{(m+1)/2}},\\
    |\partial_\lambda j_\lambda|\leq & C \frac{|x|}{\lambda^2},\\
    |\nabla^m\partial_\lambda j_\lambda|\leq& C(m)\frac{1}{\lambda^2}.
\end{align*}
Then we note that $\partial_\lambda \nabla^kz=\partial_\lambda \nabla^k\pi_\lambda+\partial_\lambda R_k$ so
by differentiating our remainders from before one can obtain the bounds in $\mathbb{D}_r$
\begin{align}
\begin{split}\label{eq: lambda z remainders}
    \partial_\lambda z=&\partial_\lambda\pi_\lambda+\mathcal{O}(\frac{|x|}{\lambda^2}),\\
    \partial_\lambda \nabla^k z=& \partial_\lambda \nabla^k \pi_\lambda+\mathcal{O}(\frac{\lambda^{k-3}}{(1+\lambda^2|x|^2)^{(k-1)/2}})\quad\text{ for }k\geq 1,\\
    \partial_\lambda\nabla^k\Delta z=&\partial_\lambda\nabla^k\Delta \pi_\lambda+\mathcal{O}(\frac{\lambda^{k-1}}{(1+\lambda^2|x|^2)^{k/2+1}})\quad\text{ for }k\geq 0.
\end{split}
\end{align}

Also by a similar analysis to before of $z$ on the annulus $\mathbb{D}_{2r}\setminus\mathbb{D}_r$ and on the rest of $\Sigma$, we obtain on $\Sigma\setminus U_r$
\begin{align}
\begin{split}\label{eq: lambda z bounds away from D_r}
    |\partial_\lambda \nabla^k z|=&\mathcal{O}(\frac{1}{\lambda^2})\quad\text{ for }k\geq 0,\\
    |\partial_\lambda \nabla^k\Delta z|=&\mathcal{O}(\frac{1}{\lambda^3})\quad\text{ for }k\geq 0.
\end{split}
\end{align}

\subsection{$\nabla_A$ derivatives}

For the $\nabla_A$ derivatives we need to take more care. In the flat case we are working on a homogeneous domain so our choice of base point does not change anything and all derivatives are 0. Now in the hyperbolic case given $A\in T_a\Sigma$ we can set $a_s=\text{Exp}_a^\Sigma(sA)$, some path in $\Sigma$ with derivative $A$ at $a$. Then
\begin{equation*}
    \nabla_Az_{\lambda,a}=\frac{\partial}{\partial s}z_{\lambda,a_s}\bigg\vert_{s=0}.
\end{equation*}
Then by explicit calculation we have inside $\mathbb{D}_r$ that, as in \cite{sharp_lowenergyalpha},
\begin{equation}
    \nabla_A z = - A^i\partial_iz-|x|^2A^i\partial_iz+2(A\cdot x)x^i\partial_i z.
\end{equation}
In general we can write for $k\geq 1$ that $\nabla^k\nabla_A z=-\nabla^kA^i\partial_iz+R_k$ where we can bound $R_k$ by
\begin{equation}\label{eq: nabla A z remainders}
    |R_k|\leq C(|x|^2|\nabla^{k+1}z|+|x|\cdot|\nabla^kz|+|\nabla^{k-1}z|)\leq C \frac{\lambda^{k-1}}{(1+\lambda^2|x|^2)^{k/2}}.
 \end{equation}
It is easy to see that on $\Sigma\setminus U_r$ we get for all $k\geq 0$
\begin{equation}\label{eq: nabla A z away from D_r}
    |\nabla^k\nabla_Az|\leq \frac{1}{\lambda}.
\end{equation}

\subsection{Terms arising from the metric}

For higher order terms in the hyperbolic case we also get terms arising from the metric.
In the hyperbolic case inside $\mathbb{D}_r$ we have, using \eqref{eq: Delta z in the ball},
\begin{align}
\begin{split}\label{eq: Delta squared z metric bound}
    (\Delta^g)^2z=&\frac{(1-|x|^2)^2}{4}\Delta(\frac{(1-|x|^2)^2}{4}\Delta z)\\
    =&(\frac{(1-|x|^2)^2}{4})^2\Delta^2z+\mathcal{O}(|x|)\nabla\Delta z+\mathcal{O}(1)\Delta z\\
    =&\frac{1}{c_\gamma^2}\Delta^2\pi_\lambda+\mathcal{O}(\frac{\lambda^2}{(1+\lambda^2|x|^2)^2}).
\end{split}
\end{align}
The final bound also holds in the flat case.
Using \eqref{eq: z bounds away from D_r}, this gives in particular the global bound
\begin{equation}\label{eq: Delta^2 z global bound}
    |\Delta^2 z|\leq C\lambda^2\rho_z^2.
\end{equation}
Differentiating through $\lambda$ and using \eqref{eq: lambda z remainders} also gives the result in $\mathbb{D}_r$
\begin{equation}\label{eq: lambda delta^2 z global bound}
    \partial_\lambda(\Delta^g)^2 z=\frac{1}{c_\gamma^2}\partial_\lambda\Delta^2\pi_\lambda+\mathcal{O}(\frac{\lambda}{(1+\lambda^2|x|^2)^2})
\end{equation}
which clearly also holds in the flat case. This and \eqref{eq: lambda z bounds away from D_r} then give the global bound
\begin{equation}
    |\partial_\lambda\Delta^2z|\leq C\lambda\rho_z^2.
\end{equation}
For $\nabla_A$, we get in $\mathbb{D}_r$ that
\begin{align*}
    (\Delta^g)^2\nabla_Az=& -A^i\Delta^2 \partial_iz+\mathcal{O}(|x|^2)\nabla\Delta^2z\\
 &\quad+\mathcal{O}(|x|)\nabla^2\Delta z+\mathcal{O}(1)\nabla^3 z+\mathcal{O}(1)\nabla^2z+\mathcal{O}(1)\nabla z\\
 =&-A^i\Delta^2 \partial_iz+\mathcal{O}(\frac{\lambda^3}{(1+\lambda^2|x|^2)^2}).
\end{align*}
Which with \eqref{eq: nabla A z away from D_r} gives us a global bound of
\begin{equation}\label{eq: delta^2 nabla A z global bound}
    |\Delta^2\nabla_A z |\leq C\lambda^3\rho_z^2.
\end{equation}

\subsection{Extra bounds}

For the lower order terms we will need some sharper estimates in order to obtain the dependence on $\mathcal{J}$. Inside of $\mathbb{D}_r$ one can view $z$ as $z=\pi_\lambda+j_\lambda^\top+K_\lambda$ where 
\begin{equation*}
    j_\lambda^\top=\text{d}P(\pi_\lambda)[j_\lambda]
\end{equation*} is the projection of $j_\lambda$ onto the plane tangent to $\pi_\lambda$ and
\begin{equation*}
    K_\lambda=z-\pi_\lambda-j_\lambda^\top=\int^1_0\text{d}P(\pi_\lambda+tj_\lambda)[j_\lambda]-\text{d}P(\pi_\lambda)[j_\lambda] \text{d}t
\end{equation*} We then have the following extra bounds on $j_\lambda$ and $K_\lambda$ in $\mathbb{D}_r$ as in \cite{sharp_lowenergyalpha}.
\begin{equation}\label{eq: extra j and K bounds}
\begin{split}
    |j_\lambda|=&\mathcal{O}(\frac{|x|}{\lambda}) ,\\ 
    |\nabla j_\lambda^\top|=&\mathcal{O}(\frac{1}{\lambda}) ,\\ 
    |\Delta j_\lambda^\top|=&\mathcal{O}(\frac{1}{(1+\lambda^2|x|^2)}) ,\\ 
    |\nabla^2 j_\lambda^\top|=&\mathcal{O}(\frac{1}{\lambda}+\frac{1}{(1+\lambda^2|x|^2)}) ,\\ 
    |\partial_\lambda j_\lambda^\top|=&\mathcal{O}(\frac{|x|}{\lambda^2})  ,\\
\end{split}
\quad\quad\quad\quad
\begin{split}
    |K_\lambda|=&\mathcal{O}(\frac{|x|^2}{\lambda^2}),\\
    |\nabla K_\lambda|=&\mathcal{O}(\frac{1}{\lambda^2}),\\
    |\Delta K_\lambda|=&\mathcal{O}(\frac{1}{\lambda^2}),\\
    |\nabla^2 K_\lambda|=&\mathcal{O}(\frac{1}{\lambda^2}),\\
    |\partial_\lambda K_\lambda|=&\mathcal{O}(\frac{|x|^2}{\lambda^3}).\\
\end{split}
\end{equation}
We also note the additional bounds
\begin{align}
\begin{split}\label{eq: sharper j and K bounds}
    |(\Delta K_\lambda)^\top|=&\mathcal{O}(\frac{1}{\lambda}\frac{|x|}{1+\lambda^2|x|^2}),\\
    |\Delta j_\lambda^\top\cdot \nabla_Aj^\top_\lambda| &=\mathcal{O}(\frac{|x|}{(1+\lambda^2|x|^2)^2}).
\end{split}
\end{align}
\begin{comment}
We also require a new bound, we start by noting that
\begin{align*}
    \Delta z=& \Delta \pi_\lambda+(\text{d}P(\pi_\lambda+j_\lambda)-\text{d}P(\pi_\lambda))[\Delta\pi_\lambda]+(\text{d}^2P(\pi_\lambda+j_\lambda)-\text{d}^2P(\pi_\lambda))[\nabla\pi_\lambda,\nabla\pi_\lambda])\\
    &+\text{d}^2P(\pi_\lambda+j_\lambda)(2[\nabla\pi_\lambda,\nabla j_\lambda]+[\nabla j_\lambda,\nabla j_\lambda])
\end{align*}
then by differentiating through $\lambda$ and \eqref{eq: lambda z remainders} we obtain
\begin{equation*}
     \Delta\partial_\lambda z= \Delta\partial_\lambda \pi_\lambda+\mathcal{O}(\frac{|x|}{\lambda^2}\rho_z^2)+2\text{d}^2P(\pi_\lambda+j_\lambda)([\nabla\partial_\lambda\pi_\lambda,\nabla j_\lambda]+[\nabla\pi_\lambda,\nabla\partial_\lambda j_\lambda]+[\nabla\partial_\lambda j_\lambda,\nabla j_\lambda]).
\end{equation*}
Now we note that using mean value inequality and the exact form of $\text{d}^2P$ on the round sphere,
\begin{align*}
    |(\text{d}^2P(\pi_\lambda+j_\lambda)[f,g])^{\top_z}|\leq&|(\text{d}^2P(\pi_\lambda)[f,g])^{\top_{\pi_\lambda}}|+C|j_\lambda|\cdot|f|\cdot|g|\\
    \leq& C(|f^{\top_{\pi_\lambda}}|\cdot|g^{\bot_{\pi_\lambda}}|+|g^{\top_{\pi_\lambda}}|\cdot|f^{\bot_{\pi_\lambda}}|+\frac{|x|}{\lambda}|f|\cdot|g|)
\end{align*}
and so analysis of the possible terms yields
\begin{equation}\label{eq: Delta lambda z tangent bound}
    ( \Delta\partial_\lambda z)^{\top_z}= (\Delta\partial_\lambda \pi_\lambda)^{\top_z}+\mathcal{O}(\frac{|x|}{\lambda^2}\rho_z^2).
\end{equation}
\end{comment}
We also note that by expanding out and using the exact form of $\text{d}^2P_{\mathbb{S}^2}$ we obtain
\begin{equation}\label{eq: Delta lambda z tangent bound}
    ( \Delta\partial_\lambda z)^{\top_z}= (\Delta\partial_\lambda \pi_\lambda)^{\top_z}+\mathcal{O}(\frac{|x|}{\lambda^2}\rho_z^2).
\end{equation}

\section{Finishing the proof of Lemma \ref{Lemma: Raw energy bound}}\label{section: proof of raw energy bound}
To finish the proof of Lemma \ref{Lemma: Raw energy bound} we need to expand the Dirichlet energy along $\mathcal{Z}$.
For $z\in\mathcal{Z}$ we claim
\begin{equation*}
    E[z]=\frac{1}{2}\int_\Sigma|\nabla z|^2=4\pi-4\pi\mathcal{J}(a)\frac{1}{\lambda^2}+\mathcal{O}(\frac{1}{\lambda^3}).
\end{equation*}
Using $\eqref{eq: naive pi derivative bounds}$ and $\eqref{eq: naive j derivative bounds}$ one can show that
\begin{align*}
    \int_{\mathbb{D}_{2r}}|\nabla z|^2=&\int_{\mathbb{D}_{2r}}|\nabla \pi|^2+2\Delta \pi\cdot j+2\nabla \pi\cdot \nabla j+|\nabla j|^2+\mathcal{O}(\frac{1}{\lambda^3})\\
    %=&\int_{\mathbb{D}_{2r}}|\nabla \pi|^2+\Delta \pi\cdot j+\Delta( \pi\cdot  j)+|\nabla j|^2+\mathcal{O}(\frac{1}{\lambda^3})\\
    =&8\pi-\frac{8\pi}{1+4\lambda^2r^2}-8\pi\mathcal{J}(a)\frac{1}{\lambda^2}\\
    &+\frac{1}{2r}\int_{\partial\mathbb{D}_{2r}}x^a(\nabla_{x^a}\pi\cdot j+\pi\cdot\nabla_{x^a}j+\nabla_aj\cdot j)+\mathcal{O}(\frac{1}{\lambda^3}).
\end{align*}
Now on $\Sigma\setminus U_{2r}$ we calculate
\begin{align*}
    \int_{\Sigma\setminus U_{2r}}|\nabla z|^2
    %=&\int_{\Sigma\setminus U_{2r}}|\nabla f|^2+\mathcal{O}(\frac{1}{\lambda^3})\\
    %=&-\frac{1}{2r}\int_{\partial\mathbb{D}_{2r}} ((x\cdot\nabla)f)\cdot f+\mathcal{O}(\frac{1}{\lambda^3})\\
    =&\frac{2\pi}{\lambda^2r^2}-\frac{1}{2r}\int_{\partial\mathbb{D}_{2r}}x^a(\nabla_{x^a}\pi\cdot j+\pi\cdot\nabla_{x^a}j+\nabla_aj\cdot j)+\mathcal{O}(\frac{1}{\lambda^3}).
\end{align*}
Summing these two equations completes the proof.

\bibliographystyle{abbrv}
\bibliography{Bibliography}

\endnotemark[1]School of Mathematics, University of Leeds, Leeds, LS2 9JT, United Kingdom\\ A.M.Roberts@leeds.ac.uk

\end{document}